\documentclass{scrartcl}

\usepackage[utf8]{luainputenc}
\usepackage[USenglish]{babel}
\usepackage{csquotes}

\usepackage[a4paper,top=27mm,bottom=20mm,inner=25mm,outer=20mm]{geometry}

\usepackage[%
  backend=bibtex,bibencoding=ascii,
  style=numeric-comp,
  giveninits=true, uniquename=init, 
  natbib=true,
  url=true,
  doi=true,
  isbn=false,
  backref=false,
  maxnames=99,
  ]{biblatex}
\usepackage{amsmath}
\allowdisplaybreaks
\usepackage{amssymb}
\usepackage{commath}
\usepackage{mathtools}
\usepackage{bbm}
\usepackage{nicefrac}
\usepackage{subdepth}

\usepackage{siunitx}

\usepackage{amsthm}
\usepackage{thmtools}
\usepackage{etoolbox}
\makeatletter
\patchcmd{\thmt@setheadstyle}
 {\bgroup\thmt@space}
 {\thmt@space}
 {}{}
\patchcmd{\thmt@setheadstyle}
 {\egroup\fi}
 {\fi}
 {}{}
\makeatother
\declaretheoremstyle[
  bodyfont=\normalfont\itshape,
  headformat=\NAME\ \NUMBER\NOTE,
]{myplain}
\declaretheoremstyle[
  headformat=\NAME\ \NUMBER\NOTE,
]{mydefinition}
\newcommand{\envqed}{{\lower-0.3ex\hbox{$\triangleleft$}}}
\declaretheorem[style=myplain,numberwithin=section]{theorem}
\declaretheorem[style=myplain,numberlike=theorem]{lemma}

\declaretheorem[style=myplain,numberlike=theorem]{corollary}
\declaretheorem[style=mydefinition,numberlike=theorem,qed=\envqed]{definition}
\declaretheorem[style=mydefinition,numberlike=theorem,qed=\envqed]{remark}

\usepackage[plainpages=false,pdfpagelabels,hidelinks,unicode]{hyperref}

\usepackage{color}
\usepackage{graphicx}
\usepackage[small]{caption}
\usepackage{subcaption}

\begingroup\expandafter\expandafter\expandafter\endgroup
\expandafter\ifx\csname pdfsuppresswarningpagegroup\endcsname\relax
\else
\fi

\usepackage{booktabs}
\usepackage{rotating}
\usepackage{multirow}

\usepackage{enumitem}
\usepackage{calc}

\usepackage{ifluatex}
\ifluatex
  \usepackage[no-math]{fontspec}
\else
  \usepackage[T1]{fontenc}
\fi
\usepackage{newpxtext,newpxmath}

\usepackage{xparse}

\let\epsilon\varepsilon
\let\phi\varphi
\let\rho\varrho

\usepackage{xspace}

\newcommand{\cf}[0]{{cf.\@}\xspace}
\newcommand{\eg}[0]{{e.g.\@}\xspace}
\newcommand{\ie}[0]{{i.e.\@}\xspace}

\newcommand{\etal}[0]{{et al.\@}\xspace}
\newcommand{\trixi}{\texttt{Trixi.jl}\xspace}

\newcommand{\e}{\mathrm{e}}

\newcommand{\R}{\mathbb{R}}

\newcommand{\const}{\mathrm{const}}

\newcommand{\dx}{\Delta x}

\newcommand{\fnum}{f^{\mathrm{num}}}

\newsavebox{\DelimiterBox}
\newlength{\DelimiterHeight}
\newlength{\DelimiterDepth}
\newsavebox{\ArgumentBox}
\newlength{\ArgumentHeight}
\newlength{\ArgumentDepth}
\newlength{\ResizedDelimiterHeight}
\newlength{\ResizedDelimiterDepth}

\newcommand{\encloseby}[3]{%
  \savebox{\ArgumentBox}{$\displaystyle #1$}%
  \settoheight{\ArgumentHeight}{\usebox{\ArgumentBox}}%
  \settodepth{\ArgumentDepth}{\usebox{\ArgumentBox}}%
  \savebox{\DelimiterBox}{#2}%
  \settoheight{\DelimiterHeight}{\usebox{\DelimiterBox}}%
  \settodepth{\DelimiterDepth}{\usebox{\DelimiterBox}}%
  \setlength{\ResizedDelimiterHeight}{%
    \maxof{1.2\ArgumentHeight}{\DelimiterHeight}%
  }
  \setlength{\ResizedDelimiterDepth}{%
    \maxof{1.2\ArgumentDepth}{\DelimiterDepth}%
  }
  \raisebox{-\ResizedDelimiterDepth}{%
    \resizebox{\width}{\ResizedDelimiterHeight+\ResizedDelimiterDepth}{%
      \raisebox{\DelimiterDepth}{#2}%
    }%
  }
  #1
  \raisebox{-\ResizedDelimiterDepth}{%
    \resizebox{\width}{\ResizedDelimiterHeight+\ResizedDelimiterDepth}{%
      \raisebox{\DelimiterDepth}{#3}%
    }%
  }
}

\ifluatex
  \newcommand{\mean}[1]{\encloseby{#1}{$\{\mkern-5mu\{$}{$\}\mkern-5mu\}$}}
  \newcommand{\prodmean}[1]{\encloseby{#1}{$(\mkern-4mu($}{$)\mkern-4mu)$}}
  \newcommand{\jump}[1]{\encloseby{#1}{$[\mkern-4mu[$}{$]\mkern-4mu]$}}
\else
  \newcommand{\mean}[1]{\encloseby{#1}{$\{\mkern-6mu\{$}{$\}\mkern-6mu\}$}}
  \newcommand{\prodmean}[1]{\encloseby{#1}{$(\mkern-3mu($}{$)\mkern-3mu)$}}
  \newcommand{\jump}[1]{\encloseby{#1}{$[\mkern-3mu[$}{$]\mkern-3mu]$}}
\fi

\newcommand{\logmean}[1]{\mean{#1}_\mathrm{log}}

\newcommand{\orcid}[1]{ORCID:~\href{https://orcid.org/#1}{#1}}
\usepackage{authblk}

\newenvironment{keywords}{\par\textbf{Key words.}}{\par}
\newenvironment{AMS}{\par\textbf{AMS subject classification.}}{\par}

\title{Preventing pressure oscillations does not fix local linear stability issues of entropy-based split-form high-order schemes}
\author[1]{Hendrik Ranocha\thanks{\orcid{0000-0002-3456-2277}}}
\author[2]{Gregor J. Gassner\thanks{\orcid{0000-0002-1752-1158}}}
\affil[1]{%
King Abdullah University of Science and Technology (KAUST),
Computer Electrical and Mathematical Science and Engineering Division (CEMSE),
Thuwal, 23955-6900, Saudi Arabia}
\affil[2]{%
Department of Mathematics and Computer Science, Center for Data and Simulation Science, University of Cologne, Germany}
\date{April 13, 2021} 

\makeatletter
\hypersetup{pdfauthor={Hendrik Ranocha, Gregor J. Gassner}}
\hypersetup{pdftitle={\@title}}
\makeatother

\begin{document}

\maketitle

\begin{abstract}
  Recently, it was discovered that the entropy-conserving/dissipative high-order split-form discontinuous Galerkin discretizations have robustness issues when trying to solve the simple density wave propagation example for the compressible Euler equations. The issue is related to missing local linear stability, i.e. the stability of the discretization towards perturbations added to a stable base flow. This is strongly related to an anti-diffusion mechanism, that is inherent in entropy-conserving two-point fluxes, which are a key ingredient for the high-order discontinuous Galerkin extension. In this paper, we investigate if pressure equilibrium preservation is a remedy to these recently found local linear stability issues of entropy-conservative/dissipative high-order split-form discontinuous Galerkin methods for the compressible Euler equations. Pressure equilibrium preservation describes the property of a discretization to keep pressure and velocity constant for pure density wave propagation. We present the full theoretical derivation, analysis, and show corresponding numerical results to underline our findings. In addition, we characterize numerical fluxes for the Euler equations that are entropy-conservative, kinetic-energy-preserving, pressure-equilibrium-preserving, and have a density flux that does not depend on the pressure. The source code to reproduce all numerical experiments presented in this article is available online (DOI: \texorpdfstring{\href{https://doi.org/10.5281/zenodo.4054366}{10.5281/zenodo.4054366}}{10.5281/zenodo.4054366}).

\end{abstract}

\begin{keywords}
  entropy conservation,
  kinetic energy preservation,
  pressure equilibrium preservation,
  compressible Euler equations,
  local linear stability,
  summation-by-parts
\end{keywords}

\begin{AMS}
  65M12,  
  65M70,  
  65M06,  
  65M60,  
  35Q35   
\end{AMS}

\section{Introduction}
\label{sec:introduction}

In recent years, discontinuous Galerkin (DG) spectral collocation methods with
summation-by-parts (SBP) property have gained a lot of traction in the high-order
community \cite{sjogreen2010skew,sjogreen2017skew,sjogreen2018high,chan2018discretely,chan2019efficient,parsani2015entropyInterfaces,parsani2015entropyWall,flad2017use},
due to the possibility to construct entropy-conservative/dissipative
\cite{tadmor1987numerical,tadmor2003entropy,lefloch2002fully,fisher2013high,ranocha2018comparison,chen2017entropy}
and/or kinetic-energy-preserving
\cite{jameson2008formulation,ranocha2018thesis,ranocha2020entropy,kuya2018kinetic}
discretizations. Such discretizations are currently successful, as they provide
strongly increased robustness for the approximation of highly non-linear problems
\cite{gassner2016split,rojas2021robustness,klose2020assessing},
in some cases even outperform DG discretization with polynomial de-aliasing
\cite{winters2018comparative}.
A key building block in these novel high-order collocation discretizations is a
special two-point flux formulation of the volume terms introduced by
LeFloch, Mercier, and Rohde for central finite differences in periodic domains
\cite{lefloch2002fully}, Fisher \etal for SBP finite differences in
bounded domains \cite{fisher2013high}, and by Carpenter \etal
for discontinuous spectral collocation schemes
\cite{carpenter2014entropy,carpenter2016towards}.

Unsurprisingly, the choice of (symmetric) two-point flux function used in the novel
volume term formulation is a key ingredient and determines
the properties of the resulting high-order discretization. It is a somewhat
surprising result that properties of the two-point fluxes used in simple
low-order finite volume formulations directly translate to the high-order
volume integral terms in this formulation. When using an entropy-conserving
two-point finite volume flux, the corresponding two-point flux volume integral
term of the DG scheme is entropy-conserving as well
\cite{lefloch2002fully,fisher2013high}.
The same is for instance true for kinetic energy preservation
\cite{gassner2016split,ranocha2018thesis}; as we will show in this paper,
it also holds for pressure equilibrium preservation.
We note that the simple arithmetic mean two-point flux function recovers
exactly the original nodal DG operator, while other choices of two-point flux
functions may result in non-linear split-form DG operators, even for linear
advection problems.

Unfortunately, it was recently discovered that the novel
entropy-conserving/dissipative (and many other split-form) DG schemes can
have stability issues \cite{gassner2020stability}.
While the DG discretization is equipped with a provably discrete entropy
inequality, it turns out that the schemes might struggle to retain
\emph{local linear stability}, \ie the linear stability of the non-linear
operator when linearized around a base-flow. Investigations of the spectrum
of the linearized high-order operators revealed modes with spurious exponential
growth, attributed to anti-diffusion of entropy-conserving two-point fluxes.
A particular striking example is given in \cite{gassner2020stability} for
the compressible Euler equations with a simple density wave
\begin{equation}
\label{eq:density_wave}
\begin{pmatrix}
    \rho(x,t) \\
    v(x,t) \\
    p(x,t)
  \end{pmatrix}
=
\begin{pmatrix}
    1 + 0.98 \sin(2\,\pi\,(x-v\,t))\\
    0.1\\
    20
  \end{pmatrix},
\end{equation}
where the density $\rho$ is variable, but the velocity $v$ and pressure $p$ are constant. Such a density wave \eqref{eq:density_wave} is a simple and
smooth exact solution to the compressible Euler equations
with perfect gas law, when equipped with appropriate (\eg periodic) boundary
conditions. Surprisingly, the entropy-conserving/dissipative DG schemes and
other split-form variants fundamentally struggle for this simple problem. It turns out that the linearized spectrum
shows spurious modes with exponential growth, that may cause fatal crashing of the simulation.

In another recent paper, Shima \etal \cite{shima2021preventing}
investigated the capability of their kinetic-energy-preserving two-point flux
to retain what they call \emph{pressure equilibrium}. Consider the
compressible Euler equations with an ideal gas law,
\begin{equation}
\label{eq:euler}
  \partial_t
  \underbrace{\begin{pmatrix}
    \rho \\
    \rho v \\
    \rho e
  \end{pmatrix}}_{= u}
  + \partial_x
  \underbrace{\begin{pmatrix}
    \rho v \\
    \rho v^2 + p \\
    (\rho e + p) v
  \end{pmatrix}}_{= f(u)}
  = 0,
\end{equation}
where $\rho e$ is the total energy, $\rho \epsilon$ the internal energy,
$\rho v^2 / 2$ the kinetic energy, and
\begin{equation}
  p
  =
  (\gamma - 1) \rho \epsilon,
  \qquad
  \rho \epsilon
  =
  \rho e - \frac{1}{2} \rho v^2.
\end{equation}
Pressure equilibrium is precisely the case, when velocity $v$ and pressure $p$
are both constant, \eg the density-wave \eqref{eq:density_wave}. We get from
the evolution of the compressible Euler equations the evolution equations of
the velocity
\begin{equation}
  \rho \partial_t v
  =
  \partial_t (\rho v) - v \partial_t \rho
  =
  - \partial_x (\rho\,v^2) - \partial_x p+ v\,\partial_x (\rho\,v),
\end{equation}
and of the pressure
\begin{equation}
\begin{aligned}
  \frac{1}{\gamma - 1} \partial_t p
  &=
  \partial_t (\rho e) - \frac{1}{2} \partial_t (\rho v^2)
  =
  - \partial_x ( (\rho e + p) v)
  - \frac{1}{2} v \partial_t (\rho v)
  - \frac{1}{2} v \rho \partial_t v
  \\
  &=
  - \frac{\gamma}{\gamma - 1} \partial_x (p v)
  - \frac{1}{2} \partial_x (\rho v^3)
  + \frac{1}{2} v \partial_x (\rho v^2 + p)
  - \frac{1}{2} v \rho \partial_t v.
\end{aligned}
\end{equation}
It follows that for constant velocity and pressure, the time derivatives
$\partial_t v = 0$ and $\partial_t p =0$, hence the coined term
\emph{pressure equilibrium}. Shima \etal \cite{shima2021preventing}
found exponential spurious growth for a similar density-wave test case when
using their kinetic-energy-preserving two-point flux \cite{kuya2018kinetic}.
When they modified the two-point flux to discretely preserve pressure
equilibrium, they could demonstrate numerically that the novel scheme robustly
solves the density-wave, even for very long simulation times.

In summary, the starting point of this paper are the works \cite{gassner2020stability,shima2021preventing} and we view the current work as a direct continuation of the analysis and discussion presented therein. This brings us directly to the research questions we are adressing in the current paper:
\begin{enumerate}[label=(RQ\arabic*), labelwidth=\widthof{\ref{RQ3}}, leftmargin=!]
  \item \label{RQ1}
  Is it possible to construct two-point flux functions that are not only
  kinetic-energy-preserving and pressure-equilibrium-preserving as the one
  proposed by Shima \etal \cite{shima2021preventing}, but also entropy-conserving
  (EC) according to Tadmor's condition \cite{tadmor1987numerical,tadmor2003entropy}?

  \item \label{RQ2}
  Is pressure equilibrium preservation a remedy for the local linear stability
  issues of the entropy-conserving/dissipative DG framework reported in
  \cite{gassner2020stability}?

  \item \label{RQ3}
  Are there entropies, such that the EC two-point fluxes and corresponding EC volume integral terms are
  locally linearly stable?
\end{enumerate}

The remainder of the paper is organized as follows: in the next section,
Section~\ref{sec:fluxes}, we investigate research question \ref{RQ1} and discuss the
construction and existence of entropy-conserving (EC), kinetic-energy-preserving
(KEP), and pressure-equilibrium-preserving (PEP) two-point flux functions.
In Section~\ref{sec:stability}, we investigate research questions \ref{RQ2} \& \ref{RQ3}
and discuss the impact of pressure equilibrium preservation on local linear
stability. As a by-product, we show that the PEP property of the two-point flux
function carries over to the high-order split-form DG scheme in the
Appendix~\ref{sec:appendix1}. In the final Section~\ref{sec:summary},
we summarize our results and collect the answers to the research questions.

\section{On the construction of EC, KEP, and PEP two-point fluxes}
\label{sec:fluxes}

\subsection{Structure preservation properties}
\label{sec:structure_preservation}

The first goal of this subsection is to collect and define the properties of
the compressible Euler equations we want to preserve with our discretization.
For the definition of two-point fluxes, it suffices to concentrate on
semi-discrete finite volume methods of the form
\begin{equation}
\label{eq:fv}
  \partial_t u_i
  + \frac{1}{\dx} \bigl(
    \underbrace{\fnum(u_{i+1}, u_{i})}_{= \fnum_+}
    - \underbrace{\fnum(u_{i}, u_{i-1})}_{= \fnum_-}
  \bigr)
  = 0.
\end{equation}
In what follows, we drop the subscript $+$ for the numerical flux function for convenience and assume an interface at location $i$ and $i+1$ if not stated otherwise.

\begin{definition}[Entropy-conservation \cite{tadmor1987numerical,tadmor2003entropy}]
  A numerical flux $\fnum$ and the corresponding finite volume method is
  EC if
  \begin{equation}
  \label{eq:ec}
    \jump{w} \cdot \fnum - \jump{\psi} = 0,
  \end{equation}
  where $w = U'$ are the entropy variables, $\psi$ is the flux potential,
  and $\jump{w} := w_{i+1} - w_i$ denotes the jump operator.
\end{definition}
Unless stated otherwise, we will use the entropy
\begin{equation}
\label{eq:U}
  U = \frac{- \rho\,s}{\gamma - 1},
  \qquad
  s = \log \frac{p}{\rho^\gamma},
\end{equation}
of the compressible Euler equations \eqref{eq:euler}, with associated entropy
variables
\begin{equation}
\label{eq:w}
  w = \left(
    \frac{\gamma}{\gamma - 1} - \frac{\log \nicefrac{p}{\rho^\gamma}}{\gamma - 1} - \frac{\rho v^2}{2 p},
    \frac{\rho v}{p},
    - \frac{\rho}{p}
  \right),
\end{equation}
and flux potential $\psi = \rho v$.

\begin{definition}[Kinetic energy preservation \cite{jameson2008formulation,ranocha2018thesis,ranocha2020entropy,kuya2018kinetic}]
  A numerical flux $\fnum = (\fnum_\rho, \fnum_{\rho v}, \fnum_{\rho e})$ and
  the corresponding finite volume method is KEP if
  \begin{equation}
  \label{eq:kep}
    \fnum_{\rho v} = \mean{v} \fnum_{\rho} + \mean{p},
  \end{equation}
  where $ \mean{p}:= (p_i + p_{i+1})/2$ denotes the arithmetic mean.
\end{definition}

\begin{definition}[Pressure equilibrium preservation]
  A numerical flux $\fnum = (\fnum_\rho, \fnum_{\rho v}, \fnum_{\rho e})$ and
  the corresponding finite volume method is PEP if
  \begin{equation}
  \label{eq:pep}
  \begin{aligned}
    \fnum_{\rho v}
    &=
    v \fnum_{\rho} + \const(p, v),
    \\
    \fnum_{\rho e}
    &=
    \frac{1}{2} v^2 \fnum_\rho + \const(p, v),
  \end{aligned}
  \end{equation}
  whenever the velocity $v$ and the pressure $p$ are constant throughout the
  domain.
\end{definition}

We motivate our definition of PEP fluxes with the following 
\begin{lemma}
\label{lem:pep-semidiscrete}
  Pressure equilibrium, \ie $p \equiv \const$, $v \equiv \const$, is preserved
  by \eqref{eq:fv} if and only if $\fnum$ is PEP.
\end{lemma}
\begin{proof}
  The semidiscrete evolution equation for the velocity is
  \begin{equation}
    \rho \partial_t v
    =
    \partial_t \rho v - v \partial_t \rho
    =
    - \frac{1}{\dx} \left(
      \fnum_{\rho v, +} - \fnum_{\rho v, -}
      - v ( \fnum_{\rho, +} - \fnum_{\rho, -} )
    \right).
  \end{equation}
  Similarly, for $\partial_t v = 0$, the pressure evolves according to
  \begin{equation}
    \frac{1}{\gamma - 1} \partial_t p
    =
    \partial_t \rho \epsilon
    =
    \partial_t \rho e - \frac{1}{2} v^2 \partial_t \rho
    =
    - \frac{1}{\dx} \left(
      \fnum_{\rho e, +} - \fnum_{\rho e, -}
      - \frac{1}{2} v^2 ( \fnum_{\rho, +} - \fnum_{\rho, -} )
    \right).
  \end{equation}
  Thus, $\partial_t v = 0$ and $\partial_t p = 0$ if and only if \eqref{eq:pep}
  is satisfied.
\end{proof}

We are ready to formulate the central theorem of this work and to give the answer to the first research question \ref{RQ1} in the following
\begin{theorem}
\label{thm:ec-kep-pep-fnumrho}
  The numerical flux of Ranocha \cite{ranocha2018thesis,ranocha2020entropy},
  \begin{equation}
  \label{eq:ec-kep-pep-fnumrho} 
  \begin{aligned}
    \fnum_{\rho}
    &=
    \logmean{\rho} \mean{v},
    \\
    \fnum_{\rho v}
    &=
    \logmean{\rho} \mean{v}^2 + \mean{p},
    \\
    \fnum_{\rho e}
    &=
    \frac{1}{2} \logmean{\rho} \mean{v} \prodmean{v \cdot v}
    + \frac{1}{\gamma - 1} \logmean{\rho} \logmean{\nicefrac{\rho}{p}}^{-1} \mean{v}
    + \prodmean{p \cdot v},
  \end{aligned}
  \end{equation}
  with logarithmic mean
  \begin{equation}
    \logmean{\rho} := \frac{\jump{\rho}}{\jump{\log \rho }},
  \end{equation}
  and product mean
  \begin{equation}
    \prodmean{a \cdot b} := \frac{a_+ b_- + a_- b_+}{2} =
    2 \mean{a} \mean{b} - \mean{a b},
  \end{equation}
  for the compressible Euler equations \eqref{eq:euler} is \textnormal{EC, KEP, PEP},
  and has a density flux $\fnum_{\rho}$ that does not depend on the pressure.
  Moreover, it is the only numerical flux with these properties for
  $v \equiv \const$.
\end{theorem}

\begin{remark}
  The motivation for the last property, \ie that the density flux does not depend
  on pressure such as \eg in the EC flux by Ismail and Roe
  \cite{ismail2009affordable}, is due to the discussion presented in
  \cite{derigs2017novel,ranocha2018comparison},
  where positivity failure could be identified for certain setups with large
  pressure jumps and constant densities.
\end{remark}

\begin{remark}
\label{rem:uniqueness-1}
  The numerical flux \eqref{eq:ec-kep-pep-fnumrho} can also be
  derived by reversing the role of energy and entropy in the
  compressible Euler equations \cite[Section~5]{ranocha2018comparison}.
  Indeed, the flux (66) of \cite{ranocha2018comparison}
  is the same as \eqref{eq:ec-kep-pep-fnumrho} developed in
  \cite[Theorem~7.8]{ranocha2018thesis}.
  This numerical flux is essentially uniquely defined by its properties,
  \cf Remark~\ref{rem:uniqueness-2}.
\end{remark}

\subsection{Proof of Theorem \ref{thm:ec-kep-pep-fnumrho}}
\label{sec:proof}

We first investigate the necessary conditions for EC and PEP and get the following 
\begin{lemma}
\label{lem:pconst-vconst-ec-kep/pep}
  For $p \equiv \const$, $v \equiv \const$, an EC numerical flux that is also
  KEP or PEP satisfies
  \begin{equation}
  \label{eq:pconst-vconst-ec-kep/pep}
    \fnum_{\rho e}
    =
    \frac{1}{2} v^2 \fnum_{\rho}
    + \frac{\gamma}{\gamma - 1} \frac{p}{\logmean{\rho}} \fnum_{\rho}.
  \end{equation}
\end{lemma}
\begin{proof}
  For $p \equiv \const$, $v \equiv \const$, the left-hand side of \eqref{eq:ec}
  reduces to
  \begin{equation}
    \jump{w} \cdot \fnum - \jump{\psi}
    =
    \left(
      \frac{\gamma}{\gamma - 1} \jump{\log \rho}
      - \frac{v^2}{2 p} \jump{\rho}
    \right) \fnum_{\rho}
    + \frac{v}{p} \jump{\rho} \fnum_{\rho v}
    - \frac{1}{p} \jump{\rho} \fnum_{\rho e}
    - v \jump{\rho}.
  \end{equation}
  Inserting $\fnum_{\rho v} = v \fnum_{\rho} + p$ from the KEP \eqref{eq:kep}
  or PEP \eqref{eq:pep} property and using the discrete chain rule
  \begin{equation}
  \label{eq:logmean}
    \jump{\log \rho} = \frac{\jump{\rho}}{\logmean{\rho}}
  \end{equation}
  results in
  \begin{multline}
    \left(
      \frac{\gamma}{\gamma - 1} \frac{\jump{\rho}}{\logmean{\rho}}
      - \frac{v^2}{2 p} \jump{\rho}
    \right) \fnum_{\rho}
    + \frac{v^2}{p} \jump{\rho} \fnum_{\rho}
    + v \jump{\rho}
    - \frac{1}{p} \jump{\rho} \fnum_{\rho e}
    - v \jump{\rho}
    \\
    =
    \frac{\gamma}{\gamma - 1} \frac{\jump{\rho}}{\logmean{\rho}} \fnum_{\rho}
    + \frac{v^2}{2 p} \jump{\rho} \fnum_{\rho}
    - \frac{1}{p} \jump{\rho} \fnum_{\rho e}.
  \end{multline}
  This expression has to vanish for arbitrary values of $\rho_\pm$ for an EC flux,
  resulting in \eqref{eq:pconst-vconst-ec-kep/pep}.
\end{proof}

\begin{lemma}
\label{lem:pconst-vconst-ec-pep}
  For $p \equiv \const$, $v \equiv \const$, an EC and PEP numerical flux
  must be of the form
  \begin{equation}
  \label{eq:pconst-vconst-ec-pep}
  \begin{aligned}
    \fnum_{\rho}
    &=
    \logmean{\rho} v,
    \\
    \fnum_{\rho v}
    &=
    \logmean{\rho} v^2 + p,
    \\
    \fnum_{\rho e}
    &=
    \frac{1}{2} \logmean{\rho} v^3 + \frac{\gamma}{\gamma - 1} p v.
  \end{aligned}
  \end{equation}
\end{lemma}
\begin{proof}
  Comparing \eqref{eq:pep} and \eqref{eq:pconst-vconst-ec-kep/pep},
  \begin{equation}
    \frac{\gamma}{\gamma - 1} \frac{p}{\logmean{\rho}} \fnum_{\rho}
  \end{equation}
  must be independent of $\rho_\pm$. Hence, $\fnum_{\rho}$ must be of the form
  $\fnum_{\rho} = \logmean{\rho} v$ for $p \equiv \const, v \equiv \const$.
  Inserting the PEP property \eqref{eq:pep} for $\fnum_{\rho v}$ results in the
  final form \eqref{eq:pconst-vconst-ec-pep}.
\end{proof}

\begin{lemma}
\label{lem:vconst-ec-pep}
  For $v \equiv \const$, an EC and PEP numerical flux for which the density flux
  does not depend on the pressure must be of the form
  \begin{equation}
  \label{eq:vconst-ec-pep}
  \begin{aligned}
    \fnum_{\rho}
    &=
    \logmean{\rho} v,
    \\
    \fnum_{\rho v}
    &=
    \logmean{\rho} v^2 + \phi(\rho_\pm, p_\pm),
    \\
    \fnum_{\rho e}
    &=
    \frac{1}{2} \logmean{\rho} v^3
    + \frac{1}{\gamma - 1} \logmean{\rho} \logmean{\nicefrac{\rho}{p}}^{-1} v
    + \phi(\rho_\pm, p_\pm) v,
  \end{aligned}
  \end{equation}
  where $\phi(\rho_\pm, p_\pm)$ is some kind of mean value depending on $\rho_\pm,
  p_\pm$ such that
  $\forall \rho_\pm, p > 0\colon \phi(\rho_+, \rho_-, p, p) = p$.
\end{lemma}
\begin{proof}
  Because of Lemma~\ref{lem:pconst-vconst-ec-pep}, the general form of dependencies
  on $\rho$ for $p \equiv \const$ are already determined. The remaining degrees
  of freedom for non-constant pressure $p$ can be described by two functions
  $\phi_{1,2}$, resulting in the numerical fluxes
  \begin{equation}
  \label{eq:vconst-ec-pep-general}
  \begin{aligned}
    \fnum_{\rho}
    &=
    \logmean{\rho} v,
    \\
    \fnum_{\rho v}
    &=
    \logmean{\rho} v^2 + \phi_1(\rho_\pm, p_\pm),
    \\
    \fnum_{\rho e}
    &=
    \frac{1}{2} \logmean{\rho} v^3 + \frac{\gamma}{\gamma - 1} \phi_2(\rho_\pm, p_\pm) v,
  \end{aligned}
  \end{equation}
  where $\phi_{1,2}$ depend on $\rho_\pm, p_\pm$ such that
  \begin{equation}
    \forall p, \rho_\pm > 0\colon \quad \phi_{1,2}(\rho_+, \rho_-, p, p) = p.
  \end{equation}
  Inserting this form of the numerical flux in the left-hand side of \eqref{eq:ec}
  for $v \equiv \const$ results in
  \begin{equation}
  \begin{aligned}
    \jump{w} \cdot \fnum - \jump{\psi}
    &=
    \left(
      \jump{\log \rho}
      + \frac{1}{\gamma - 1} \jump{\log \frac{\rho}{p}}
      - \frac{v^2}{2} \jump{\frac{\rho}{p}}
    \right) \fnum_{\rho}
    + v \jump{\frac{\rho}{p}} \fnum_{\rho v}
    \\&\quad
    - \jump{\frac{\rho}{p}} \fnum_{\rho e}
    - v \jump{\rho}
    \\
    &=
    v \jump{\rho}
    + \frac{v}{\gamma - 1} \frac{\jump{\nicefrac{\rho}{p}}}{\logmean{\nicefrac{\rho}{p}}} \logmean{\rho}
    - \frac{v^3}{2} \jump{\frac{\rho}{p}} \logmean{\rho}
    + v^3 \jump{\frac{\rho}{p}} \logmean{\rho}
    \\&\quad
    + v \jump{\frac{\rho}{p}} \phi_1
    - \frac{v^3}{2} \jump{\frac{\rho}{p}} \logmean{\rho}
    - \frac{\gamma}{\gamma - 1} v \jump{\frac{\rho}{p}} \phi_2
    - v \jump{\rho}
    \\
    &=
    \frac{v}{\gamma - 1} \frac{\logmean{\rho}}{\logmean{\nicefrac{\rho}{p}}} \jump{\frac{\rho}{p}}
    + v \jump{\frac{\rho}{p}} \phi_1
    - \frac{\gamma}{\gamma - 1} v \jump{\frac{\rho}{p}} \phi_2.
  \end{aligned}
  \end{equation}
  Since this has to vanish for arbitrary $\rho_\pm, p_\pm, v$,
  \begin{equation}
    \phi_2
    =
    \frac{1}{\gamma} \frac{\logmean{\rho}}{\logmean{\nicefrac{\rho}{p}}}
    + \frac{\gamma - 1}{\gamma} \phi_1.
    \qedhere
  \end{equation}
\end{proof}

Having established the lemmata above, we are prepared to prove
Theorem~\ref{thm:ec-kep-pep-fnumrho}.
\begin{proof}[Proof of Theorem~\ref{thm:ec-kep-pep-fnumrho}]
  The KEP \eqref{eq:kep} property is satisfied by construction. Moreover, the numerical
  flux for the total energy satisfies the PEP property \eqref{eq:pep}, since
  it can be written as
  \begin{equation}
  \begin{aligned}
    \fnum_{\rho e}
    &=
    \frac{1}{2} \logmean{\rho} \mean{v} \prodmean{v \cdot v}
    + \frac{1}{\gamma - 1} \logmean{\nicefrac{1}{p}}^{-1} \mean{v}
    + \prodmean{p \cdot v}
    \\&\quad
    + \frac{1}{\gamma - 1} \left(
      \logmean{\rho} \logmean{\nicefrac{\rho}{p}}^{-1} - \logmean{\nicefrac{1}{p}}^{-1}
    \right) \mean{v},
  \end{aligned}
  \end{equation}
  where
  \begin{equation}
    \frac{1}{\gamma - 1} \left(
      \logmean{\rho} \logmean{\nicefrac{\rho}{p}}^{-1} - \logmean{\nicefrac{1}{p}}^{-1}
    \right) \mean{v}
    =
    0
  \end{equation}
  whenever $p$ is constant.
  Finally, the flux is EC as shown in \cite{ranocha2018thesis,ranocha2020entropy}.
  It is the only numerical flux with all these properties for $v \equiv \const$,
  since the KEP property \eqref{eq:kep} requires the pressure mean in
  \eqref{eq:vconst-ec-pep} to be $\phi = \mean{p}$.
\end{proof}

\begin{remark}
  The pressure mean in the momentum flux is determined uniquely by the KEP
  property \eqref{eq:kep}, resulting in a pressure mean depending on the
  density in the energy flux. As required by the PEP property \eqref{eq:pep},
  this dependency occurs only for non-constant pressure. However, such a
  mixed dependency on $\rho$ and $p$ of an approximation to the pressure is
  necessary for EC and PEP fluxes because of Lemma~\ref{lem:vconst-ec-pep}.
\end{remark}

We have obtained a complete characterization of numerical fluxes
for the compressible Euler equations that are EC, KEP, PEP, and
have a density flux $\fnum_\rho$ that does not depend on the
pressure for $v \equiv \const$ in Theorem~\ref{thm:ec-kep-pep-fnumrho}.
The analogous characterization for $p \equiv \const$ is a bit more
involved and leaves a degree of freedom.

\begin{lemma}
\label{lem:pconst-ec-kep-pep}
  For fixed $p \equiv \const$, a (symmetric) EC, KEP, and PEP
  numerical flux must be of the form
  \begin{equation}
  \label{eq:pconst-ec-kep-pep}
  \begin{aligned}
    \fnum_{\rho}
    &=
    \logmean{\rho} \mean{v}
    + \chi(\rho_\pm, v_\pm),
    \\
    \fnum_{\rho v}
    &=
    \logmean{\rho} \mean{v}^2
    + p
    + \mean{v} \chi(\rho_\pm, v_\pm),
    \\
    \fnum_{\rho e}
    &=
    \frac{1}{2} \logmean{\rho} \mean{v} \prodmean{v \cdot v}
    + \frac{\gamma}{\gamma - 1} p \mean{v}
    \\
    &\quad
    +\left(
      \frac{1}{2} \prodmean{v \cdot v}
      + \frac{\gamma}{\gamma - 1} \frac{p}{\logmean{\rho}}
    \right) \chi(\rho_\pm, v_\pm),
  \end{aligned}
  \end{equation}
  where $\chi$ is a function depending on $\rho_\pm, v_\pm$
  (symmetrically with respect to $\pm$) such that
  $\forall \rho_\pm, v \colon \chi(\rho_+, \rho_-, v, v) = 0$.
\end{lemma}
\begin{proof}
  Because of consistency, every numerical flux can be written as the
  sum of a given numerical flux and a perturbation $\chi$ that is
  consistent with zero.
  Using \eqref{eq:ec-kep-pep-fnumrho} as baseline flux
  for fixed $p \equiv \const$, every numerical flux can be written as
  \begin{equation}
  \label{eq:pconst-ec-kep-pep-ansatz}
  \begin{aligned}
    \fnum_{\rho}
    &=
    \logmean{\rho} \mean{v}
    + \chi_{\rho}(\rho_\pm, v_\pm),
    \\
    \fnum_{\rho v}
    &=
    \logmean{\rho} \mean{v}^2
    + p
    + \chi_{\rho v}(\rho_\pm, v_\pm),
    \\
    \fnum_{\rho e}
    &=
    \frac{1}{2} \logmean{\rho} \mean{v} \prodmean{v \cdot v}
    + \frac{\gamma}{\gamma - 1} p \mean{v}
    + \chi_{\rho e}(\rho_\pm, v_\pm),
  \end{aligned}
  \end{equation}
  where $\forall \rho, v \colon \chi_{\rho}(\rho, \rho, v, v) =
  \chi_{\rho v}(\rho, \rho, v, v) = \chi_{\rho e}(\rho, \rho, v, v) = 0$.
  Kinetic energy preservation \eqref{eq:kep} requires
  $\chi_{\rho v} = \mean{v} \chi_{\rho}$. Since the chosen baseline
  numerical flux \eqref{eq:ec-kep-pep-fnumrho} is EC, requiring
  entropy conservation for the perturbed numerical flux yields
  \begin{equation}
  \begin{aligned}
    0
    &=
    \jump{w} \cdot \fnum - \jump{\psi}
    =
    \left(
      \frac{\gamma}{\gamma - 1} \jump{\log \rho}
      - \frac{1}{2 p} \jump{\rho v^2}
    \right) \chi_{\rho}
    + \frac{1}{p} \jump{\rho v} \mean{v} \chi_{\rho}
    - \frac{1}{p} \jump{\rho} \chi_{\rho e}
    \\
    &=
    \left(
      \frac{\gamma}{\gamma - 1} \frac{\jump{\rho}}{\logmean{\rho}}
      - \frac{1}{2 p} \mean{v^2} \jump{\rho}
      + \frac{1}{p} \mean{v}^2  \jump{\rho}
    \right) \chi_{\rho}
    - \frac{1}{p} \jump{\rho} \chi_{\rho e}
    \\
    &=
    \left(
      \frac{1}{2 p} \prodmean{v \cdot v} \jump{\rho}
      + \frac{\gamma}{\gamma - 1} \frac{\jump{\rho}}{\logmean{\rho}}
    \right) \chi_{\rho}
    - \frac{1}{p} \jump{\rho} \chi_{\rho e}.
  \end{aligned}
  \end{equation}
  Hence,
  \begin{equation}
    \chi_{\rho e}(\rho_\pm, v_\pm)
    =
    \left(
      \frac{1}{2} \prodmean{v \cdot v}
      + \frac{\gamma}{\gamma - 1} \frac{p}{\logmean{\rho}}
    \right) \chi_{\rho}(\rho_\pm, v_\pm).
  \end{equation}
  Pressure equilibrium preservation \eqref{eq:pep} requires
  $\chi_{\rho e} = \frac{1}{2} v^2 \chi_\rho + \const(p, v)$ for
  $v \equiv \const$.
  The first term
  $\frac{1}{2} \prodmean{v \cdot v}  \chi_{\rho}$
  satisfies
  this requirement. However, the second term
  $\frac{\gamma}{\gamma - 1} \frac{p}{\logmean{\rho}} \chi_{\rho}$
  fits if and only if
  $\forall \rho_\pm, v \colon \chi(\rho_+, \rho_-, v, v) = 0$.
\end{proof}

\begin{remark}
\label{rem:uniqueness-2}
  Extending the numerical fluxes \eqref{eq:pconst-ec-kep-pep}
  developed for fixed pressure $p \equiv \const$ to general variable
  pressures results in a pressure-dependent density flux unless the
  perturbation vanishes, \ie $\chi = 0$.
  Thus, the numerical flux \eqref{eq:ec-kep-pep-fnumrho} is also
  unique for general velocities $v$ in the class of continuous numerical
  fluxes with the properties given in Theorem~\ref{thm:ec-kep-pep-fnumrho}.
\end{remark}
\begin{proof}
  Using the ansatz \eqref{eq:pconst-ec-kep-pep-ansatz} for
  a general pressure $p$ yields
  \begin{equation}
  \begin{aligned}
    \fnum_{\rho}
    &=
    \logmean{\rho} \mean{v}
    + \chi_{\rho}(\rho_\pm, v_\pm),
    \\
    \fnum_{\rho v}
    &=
    \logmean{\rho} \mean{v}^2
    + \mean{p}
    + \chi_{\rho v}(\rho_\pm, v_\pm, p_\pm),
    \\
    \fnum_{\rho e}
    &=
    \frac{1}{2} \logmean{\rho} \mean{v} \prodmean{v \cdot v}
    + \frac{1}{\gamma - 1} \logmean{\rho} \logmean{\nicefrac{\rho}{p}}^{-1} \mean{v}
    + \prodmean{p \cdot v}
    \\&\quad
    + \chi_{\rho e}(\rho_\pm, v_\pm, p_\pm),
  \end{aligned}
  \end{equation}
  where $\forall \rho, v, p \colon \chi_{\rho}(\rho, \rho, v, v) =
  \chi_{\rho v}(\rho, \rho, v, v, p, p) = \chi_{\rho e}(\rho, \rho, v, v, p, p) = 0$.
  Kinetic energy preservation \eqref{eq:kep} requires again
  $\chi_{\rho v} = \mean{v} \chi_{\rho}$. Requiring entropy
  conservation additionally yields
  \begin{align*}
    0
    &=
    \jump{w} \cdot \fnum - \jump{\psi}
    =
    \left(
      \jump{\log \rho}
      + \frac{1}{\gamma - 1} \jump{\log \nicefrac{\rho}{p}}
      - \frac{1}{2} \jump{\nicefrac{\rho v^2}{p}}
    \right) \chi_{\rho}
    \\&\quad
    \stepcounter{equation}\tag{\theequation}
    + \jump{\nicefrac{\rho v}{p}} \mean{v} \chi_{\rho}
    - \jump{\nicefrac{\rho}{p}} \chi_{\rho e}
    \\
    &=
    \left(
      \frac{\jump{\rho}}{\logmean{\rho}}
      + \frac{1}{\gamma - 1} \frac{\jump{\nicefrac{\rho}{p}}}{\logmean{\nicefrac{\rho}{p}}}
      - \frac{1}{2} \mean{v^2} \jump{\nicefrac{\rho}{p}}
      + \jump{\nicefrac{\rho}{p}} \mean{v}^2
    \right) \chi_{\rho}
    - \jump{\nicefrac{\rho}{p}} \chi_{\rho e}.
  \end{align*}
  For arbitrary $\rho_\pm, v_\pm$, choosing $p_\pm$ such that
  $\jump{\nicefrac{\rho}{p}} = 0$ requires $\chi_{\rho} = 0$.
  Hence, the perturbation $\chi$ must vanish if the density flux does
  not depend on the pressure.
\end{proof}

\subsection{The KEP and PEP two-point flux of Shima et al.}

Shima \etal \cite{shima2021preventing} introduced a modification to their
KEP flux \cite{kuya2018kinetic} and constructed a KEP flux with the PEP property,
\begin{equation}
\label{eq:shima_etal}
\begin{aligned}
  \fnum_{\rho}
  &=
  \mean{\rho} \mean{v},
  \\
  \fnum_{\rho v}
  &=
  \mean{\rho} \mean{v}^2 + \mean{p},
  \\
  \fnum_{\rho e}
  &=
  \frac{1}{2} \mean{\rho} \mean{v} \prodmean{v \cdot v}
  + \frac{1}{\gamma - 1} \mean{p} \mean{v}
  + \prodmean{p \cdot v}.
\end{aligned}
\end{equation}
We note that the density flux and the general structure of the momentum and energy fluxes
is very closely related to Ranocha's two-point flux \eqref{eq:ec-kep-pep-fnumrho},
except for the EC property, because Shima \etal use the arithmetic mean in
the density flux instead of the logarithmic mean.
Although the numerical flux \eqref{eq:shima_etal} is not EC, it has four
desirable properties, namely KEP, PEP, and a pressure-independent density flux. As we realize later in Section~\ref{sec:stability}, the fourth desirable property is the arithmetic mean of the density in the density flux function, as it enhances robustness for the density wave propagation. 

In their paper, Shima \etal demonstrate numerically very good robustness of their novel KEP and PEP discretization,
even for highly non-linear problems such as underresolved turbulence. Hence, an interesting question is whether
there is an entropy function for the compressible Euler equations such that the
two-point flux function of Shima \etal with the arithmetic mean happens to be
an EC flux. This would be a possible explanation of the improved numerical
robustness of this flux for non-linear problems. To partially answer this question, we consider next the
family of entropy functions introduced by Harten \cite{harten1983symmetric}.

Harten \cite{harten1983symmetric} discovered the family of entropy functions
for the Euler equations \eqref{eq:euler}
\begin{equation}
\label{eq:U-Harten-h}
  U = -\rho h(s),
  \qquad
  s = \log \frac{p}{\rho^\gamma},
\end{equation}
where $h$ is a sufficiently smooth function satisfying
\begin{equation}
\label{eq:h-conditions-Harten}
  \frac{h''(s)}{h'(s)} < \frac{1}{\gamma},
\end{equation}
to ensure convexity of the entropy function $U$, equation \eqref{eq:U-Harten-h}. In particular, Harten discovered
the one-parameter family
\begin{equation}
\label{eq:U-Harten-alpha}
  U = -\rho h(s),
  \qquad
  h(s)
  =
  \frac{\gamma + \alpha}{\gamma - 1} \e^{s / (\gamma + \alpha)}
  =
  \frac{\gamma + \alpha}{\gamma - 1} (p / \rho^\gamma)^{1 / (\gamma + \alpha)},
  \qquad
  \alpha > 0.
\end{equation}
Up to now, we considered the entropy \eqref{eq:U} given by $h(s) \propto s$ above,
since it is the only convex entropy \eqref{eq:U-Harten-h} which symmetrizes the
compressible Navier-Stokes equations with heat flux \cite{hughes1986new}.
Nevertheless, it is interesting to know whether there are other entropies
\eqref{eq:U-Harten-h} of the compressible Euler equations \eqref{eq:euler} that
result in a corresponding EC numerical density flux $\fnum_{\rho}$, where the mean value of the density is arithmetic.

Following the approach used in Section~\ref{sec:proof}, we will make use of
the entropy variables
\begin{equation}
\label{eq:w-Harten-h}
  w = \frac{(\gamma - 1) h'(s)}{p} \left(
    - \frac{1}{2} \rho v^2 - \frac{p}{\gamma - 1} \biggl( \frac{h(s)}{h'(s)} - \gamma \biggr),
    \rho v,
    - \rho
  \right),
\end{equation}
and the flux potential
\begin{equation}
\label{eq:psi-Harten-h}
  \psi = (\gamma - 1) h'(s) \rho v,
\end{equation}
associated with the entropy \eqref{eq:U-Harten-h}.
Lemma~\ref{lem:pconst-vconst-ec-kep/pep} is a special case of

\begin{lemma}
\label{lem:pconst-vconst-ec-kep/pep-Harten-h}
  For $p \equiv \const$, $v \equiv \const$, an EC numerical flux for the
  entropy \eqref{eq:U-Harten-h} that is also KEP or PEP satisfies
  \begin{equation}
  \label{eq:pconst-vconst-ec-kep/pep-Harten-h}
    \fnum_{\rho e}
    =
    \frac{1}{2} v^2 \fnum_{\rho}
    + \frac{p}{\gamma - 1} \frac{\jump{\gamma h' - h}}{\jump{\rho h'}}  \fnum_{\rho}.
  \end{equation}
\end{lemma}
\begin{proof}
  For $p \equiv \const$, $v \equiv \const$, the left-hand side of \eqref{eq:ec}
  reduces to
  \begin{multline}
    \jump{w} \cdot \fnum - \jump{\psi}
    =
    \left(
      - \frac{1}{2} v^2 \frac{\gamma - 1}{p} \jump{\rho h'}
      - \jump{h}
      + \gamma \jump{h'}
    \right) \fnum_{\rho}
    \\
    + v \frac{\gamma - 1}{p} \jump{\rho h'} \fnum_{\rho v}
    - \frac{\gamma - 1}{p} \jump{\rho h'} \fnum_{\rho e}
    - (\gamma - 1) v \jump{\rho h'}.
  \end{multline}
  This term vanishes if and only if
  \begin{multline}
    0 =
    - \frac{1}{2} v^2 \jump{\rho h'} \fnum_{\rho}
    - \frac{1}{\gamma - 1} p \jump{h} \fnum_{\rho}
    + \frac{\gamma}{\gamma - 1} p \jump{h'} \fnum_{\rho}
    \\
    + v \jump{\rho h'} \fnum_{\rho v}
    - \jump{\rho h'} \fnum_{\rho e}
    - p v \jump{\rho h'}.
  \end{multline}
  Inserting $\fnum_{\rho v} = v \fnum_{\rho} + p$ from the KEP \eqref{eq:kep}
  or PEP \eqref{eq:pep} property results in
  \begin{equation}
    0 =
    \frac{1}{2} v^2 \jump{\rho h'} \fnum_{\rho}
    - \frac{1}{\gamma - 1} p \jump{h} \fnum_{\rho}
    + \frac{\gamma}{\gamma - 1} p \jump{h'} \fnum_{\rho}
    - \jump{\rho h'} \fnum_{\rho e}.
  \end{equation}
  This expression has to vanish for arbitrary values of $\rho_\pm$ for an EC flux,
  resulting in \eqref{eq:pconst-vconst-ec-kep/pep-Harten-h}.
\end{proof}

Comparing \eqref{eq:pep} and \eqref{eq:pconst-vconst-ec-kep/pep-Harten-h},
a PEP flux that is also EC for \eqref{eq:U-Harten-h} must contain an average
of the density proportional to
\begin{equation}
\label{eq:Harten-h-rho-average}
  \frac{ \jump{\rho h'} }{ \jump{\gamma h' - h} }.
\end{equation}
In general, \eqref{eq:Harten-h-rho-average} is not the arithmetic mean of $\rho_\pm$.
It becomes the linear mean proportional to $\mean{\rho}$ for $p \equiv \const$
for the choice of $h$ as in \eqref{eq:U-Harten-alpha} with $\alpha = - 2 \gamma$. However, in this case, the resulting $U$ is not convex anymore. Hence, entropy-conservative
and pressure equilibrium preserving numerical fluxes for the compressible
Euler equations have to use nonlinear means of the density and we have
demonstrated that the two-point flux of Shima \etal is not related to one
of Harten's entropies.

\begin{corollary}
\label{cor:EC-KEP-PEP-arithmetic}
  There is no Harten entropy pair for the compressible Euler equations
  such that a corresponding EC two-point flux with the KEP and PEP property
  uses the arithmetic mean of the density in the density flux.
\end{corollary}

So far, we have not found any evidence, that there is another strictly convex entropy pair for which the EC flux with KEP and PEP might have an arithmetic mean and thus have the conjecture, that there is none.

Finally, it is interesting to see whether the arithmetic mean can be used in the density flux of an EC flux if the
additional constraints are relaxed by not requiring the KEP/PEP property
anymore. Considering again the family of entropies \eqref{eq:U-Harten-h}, we consider the case $v \equiv \const$,
$h'(s) \rho / p \equiv \const$. Inserting the entropy variables \eqref{eq:w-Harten-h} into the EC condition
\eqref{eq:ec} results in
\begin{equation}
\label{eq:Harten-h-EC}
  0
  =
  \jump{w} \cdot \fnum - \jump{\psi}
  =
  - \jump{h(s) - \gamma h'(s)} \fnum_{\rho}
  - (\gamma - 1) \jump{h'(s) \rho v}.
\end{equation}
Hence, the density flux must again contain an average of the density
proportional to \eqref{eq:Harten-h-rho-average}, but for the case of
$h'(s) \rho  p \equiv \const$ instead of the case $p \equiv \const$
discussed above.
For a given entropy such as \eqref{eq:U} or \eqref{eq:U-Harten-alpha}, it is
easy to solve $h'(s_+) \rho_+ / p_+ = h'(s_-) \rho_- / p_-$ for $\rho_+$ and
pick values of $p_\pm, \rho_-$  such that \eqref{eq:Harten-h-EC} is not
satisfied by $\fnum_{\rho} = \mean{\rho} v$. Hence, we arrive at

\begin{corollary}
\label{cor:EC-arithmetic}
  There is no Harten entropy pair for the compressible Euler equations
  such that a corresponding EC two-point flux uses the arithmetic mean of density
  in the density flux.
\end{corollary}


\begin{remark}
  We refer to an entropy pair if the entropy function is strictly convex
  (resulting in an invertible transformation from the conserved variables
  to the entropy variables). Linear functionals of the conserved variables
  are of course non-strictly convex and can be combined with a density flux
  using the arithmetic mean.
\end{remark}

\section{On local linear stability of EC schemes with the PEP property}
\label{sec:stability}

In this section, we consider the (local) linear stability \cite{gassner2020stability} of high-order
discretizations based on two-point fluxes. The extension to high-order
accuracy is relatively straight forward when assuming the SBP property.
Several classes of numerical methods can be formulated via (periodic) SBP
operators, including finite difference \cite{kreiss1974finite,strand1994summation},
finite volume \cite{nordstrom2001finite,nordstrom2003finite},
continuous Galerkin \cite{hicken2016multidimensional,hicken2020entropy},
discontinuous Galerkin \cite{gassner2013skew},
and flux reconstruction methods \cite{ranocha2016summation}.
A brief review how to formulate these methods in the SBP framework with
application to structure-preserving numerical methods can be found in
\cite{ranocha2021broad}.
Further details and background information about SBP methods can be found \eg in
the review articles \cite{fernandez2014review,svard2014review}.

Building upon earlier works such as \cite{lefloch2002fully,sjogreen2010skew},
Fisher and Carpenter \cite{fisher2013high} created conservative high-order semi-discretizations of
hyperbolic conservation laws using a special class of two-point numerical
fluxes. The final extension to general symmetric numerical fluxes was obtained
in \cite{chen2017entropy,ranocha2018comparison,gassner2016split} and will be
recalled briefly below.

\subsection{Numerical investigation of the robustness of the split-form DG scheme}
\label{sec:num_inv}

In this part, we consider the split-form DG approximation with the numerical
fluxes discussed above and apply them to solve a simple density wave problem. 
In particular, we compare the results when using a central flux with arithmetic means, Ranocha's
two-point flux function \eqref{eq:ec-kep-pep-fnumrho} that is EC, KEP and PEP,
and the two-point flux by Shima \etal \eqref{eq:shima_etal} that is KEP and PEP.

Following \cite{gassner2020stability}, we consider the two-dimensional
compressible Euler equations with the initial condition
\begin{equation}
\label{eq:Euler-density-wave}
  \rho = 1 + 0.98 \sin(2\pi (x_1 + x_2)),
  \quad
  v_1 = 0.1,
  \quad
  v_2 = 0.2,
  \quad
  p = 20,
  \quad x \in [-1, 1]^2,
\end{equation}
and fully periodic boundary conditions.
We use the split-form DG methods with spectral collocation on Legendre-Gauss-Lobatto nodes, with a polynomial degree of $N = 5$ on a grid with $4 \times 4$ elements implemented in the open source code \trixi
\cite{schlottkelakemper2020trixi,schlottkelakemper2020purely}.
The semi-discretizations are integrated in time using the fourth-order, five-stage,
low-storage Runge-Kutta method of \cite{carpenter1994fourth} with a relative CFL number
$\mathtt{cfl} = 0.05$, which then gets additionally scaled by the choice of polynomial degree $N$, as is common for DG. This ensures a negligible impact of the time integrator
on the numerical solution. 

\begin{remark}
The numerical methods applied in this article are written in Julia
\cite{bezanson2017julia}. The plots are created using Matplotlib
\cite{hunter2007matplotlib}.
The source code necessary to reproduce all results shown in this article
is available online \cite{ranocha2020ec-kep-pep-repro}.
We use the numerically stable evaluation of the logarithmic mean proposed
in \cite{ismail2009affordable} in \trixi.
\end{remark}

The simulation with the pure central approximation with the arithmetic mean flux
is stable for all times for the density wave \eqref{eq:Euler-density-wave}.
We emphasize that we use the central flux for both, the split-form
volume integral and the surface integral fluxes --- so there is no added
numerical dissipation. This shows that in principle, the problem is very
well resolved by the chosen DG discretization. We can further
numerically confirm, that the modification of Shima \etal \cite{shima2021preventing}
gives a high-order split-form DG discretization, that is able to robustly
run this test case for very long integration times ($t > 100$,
corresponding to more than \num{935000} time steps). Hence, at first, it
seems that the added PEP property indeed solves the robustness
issue. However, in accordance with the findings of \cite{gassner2020stability}
for other EC fluxes, the high-order DG discretization with Ranocha's
EC flux \eqref{eq:ec-kep-pep-fnumrho} crashes because of negative
density already at $t \approx 0.55$. We note that Ranocha's flux is
KEP and PEP, \ie, it preserves the pressure equilibrium by construction
and there are no fluctuations in velocity and pressure!

Following these numerical results, we can already answer our second
research question and state, that, unfortunately, the answer to (RQ2)
is \emph{no}, the PEP property is not a remedy for the stability issues
of the EC fluxes, as the simple density wave problem still crashes after
very short simulation times. Note, that these findings are not sensitive
to the choice of the $\mathtt{cfl}$ number or the time integration
method.

However, it is interesting that the Shima \etal flux can indeed robustly
run the density wave test case. Because of the PEP property, both the
flux \eqref{eq:shima_etal} of Shima \etal and the flux
\eqref{eq:ec-kep-pep-fnumrho} of Ranocha reduce the density wave
\eqref{eq:density_wave} for the compressible Euler equations to four
linear advection equations. The main difference between the Shima
\etal flux and Ranocha's flux is the mean values of the density used in
the density flux: The former uses the arithmetic mean to approximate
the linear advection, the latter uses the logarithmic mean to discretize
the linear advection.

\subsection{Stability for linear advection: Investigation of the impact of the choice of the mean value}
\label{sec:linadv-mean-values}

In this subsection, we focus on how the choice of the mean values in a
split-form approximation of the linear advection equation influences
the stability. For this purpose, we consider periodic high-order SBP
discretizations of the linear advection equation. We use
\begin{definition}
  A periodic SBP operator consisting of
  a derivative matrix $D$ approximating the first derivative as
  $D u \approx \partial_x u$,
  and a mass matrix $M$ approximating the $L^2$ scalar product via
  $u^T M v \approx \int u v$,
  such that
  \begin{equation}
  \label{eq:SBP-periodic}
    M D + D^T M = 0.
  \end{equation}
\end{definition}
Given a SBP operator $D$ with symmetric mass matrix $M$ and a symmetric
two-point numerical flux $\fnum$ for the hyperbolic conservation law
\begin{equation}
\label{eq:cl-1D}
  \partial_t u + \partial_x f(u) = 0,
\end{equation}
the semi-discretization
\begin{equation}
\label{eq:cl-1D-SBP}
  \partial_t u_i + \sum_{l} 2 D_{i,l} \fnum(u_i, u_l) = 0
\end{equation}
is a conservative approximation of \eqref{eq:cl-1D} with at least the same order
of accuracy as the SBP operator
\cite{fisher2013high,chen2017entropy,ranocha2018comparison}.
Moreover, the semi-discretization conserves the entropy $U$ of \eqref{eq:cl-1D}
if the numerical flux is entropy-conservative for that entropy $U$
\cite{fisher2013high}.

Following \cite{gassner2020stability}, we compute the spectrum of a
semi-discretization \eqref{eq:cl-1D-SBP} of the linear advection equation
\begin{equation}
\begin{aligned}
  \partial_t u(t, x) + \partial_x u(t, x) &= 0,\\
  u(0, x) &= 2 + 1.9 \sin(\pi x),
\end{aligned}
\end{equation}
in a periodic domain $x \in [0, 2]$. The scheme \eqref{eq:cl-1D-SBP} is implemented
in Julia \cite{bezanson2017julia} and the Jacobian of the semi-discretization
is computed via forward-mode automatic differentiation (AD) \cite{revels2016forward}. Note that AD
is not necessary if a linear numerical flux is used for this linear PDE.
However, we are also interested in nonlinear numerical fluxes,
involving \eg the logarithmic mean. For nonlinear discretizations,
it is not as trivial to compute the Jacobian and AD becomes a
valuable tool.

We note that if the numerical flux is chosen as the arithmetic mean, $\fnum = \mean{u}$,
the semi-discretization is linear and skew-symmetric (with respect to the scalar
product induced by the mass matrix $M$). This semi-discretization conserves the $L^2$ entropy $U(u) = u^2 / 2$.
Hence, as expected, all eigenvalues are purely imaginary in our numerical test. 

In contrast, choosing the numerical flux as the logarithmic mean
$\fnum = \logmean{u}$ results in a nonlinear semi-discretization, which
conserves the entropy $U(u) = u \log u - u$ with entropy flux $F(u) = u \log u - u$.
Indeed, the corresponding entropy variables are $w(u) = U'(u) = \log(u)$ and
the flux potential is $\psi(u) = u$. Hence, the associated entropy-conservative
numerical flux is $\fnum(u) = \jump{\psi(u)} / \jump{w(u)} = \jump{u} / \jump{\log u}
= \logmean{u}$. The resulting spectra of the Jacobian of the semi-discretization with logarithmic mean values
are shown in Figure~\ref{fig:spectra} for different choices of SBP operators.
\begin{figure}[!tp]
\centering
  \begin{subfigure}{\textwidth}
    \centering
    \begin{subfigure}{0.49\textwidth}
    \centering
      \includegraphics[width=\textwidth]{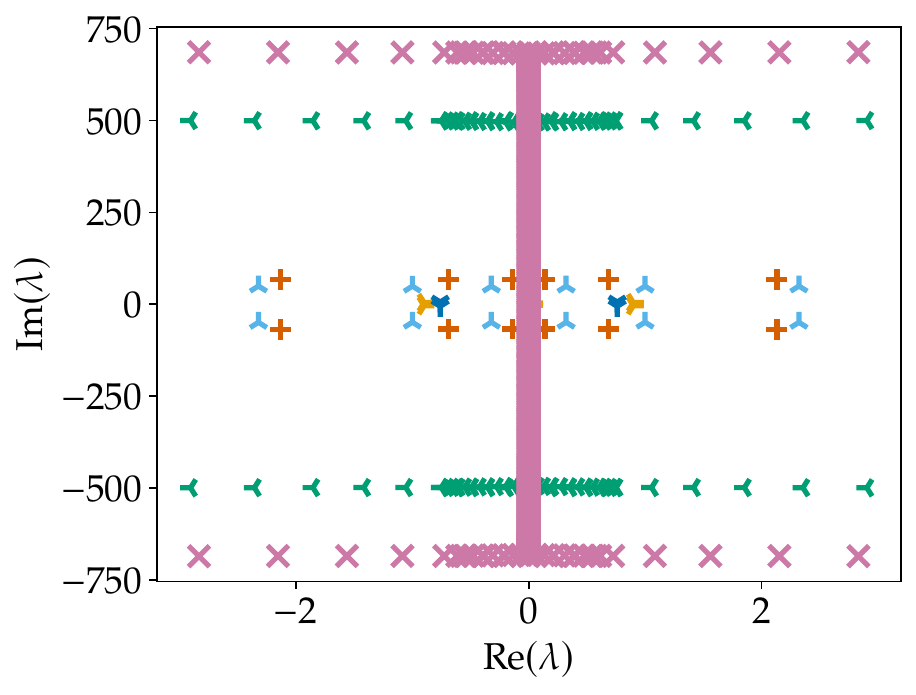}
    \end{subfigure}%
    \hspace*{\fill}
    \begin{subfigure}{0.49\textwidth}
    \centering
      \includegraphics[width=\textwidth]{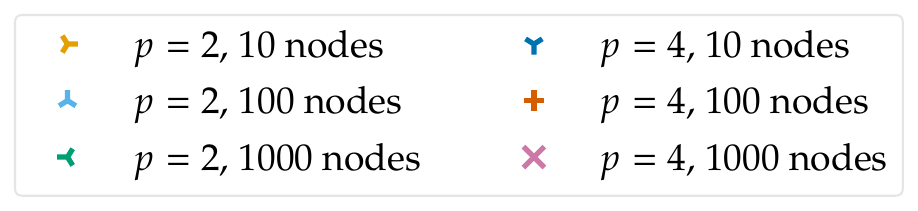}
    \end{subfigure}%
    \caption{Finite difference methods of order $p \in \{2, 4\}$ using different
             numbers of nodes.}
    \label{fig:spectra-FD}
  \end{subfigure}%
  \\
  \begin{subfigure}{\textwidth}
    \centering
    \begin{subfigure}{0.49\textwidth}
    \centering
      \includegraphics[width=\textwidth]{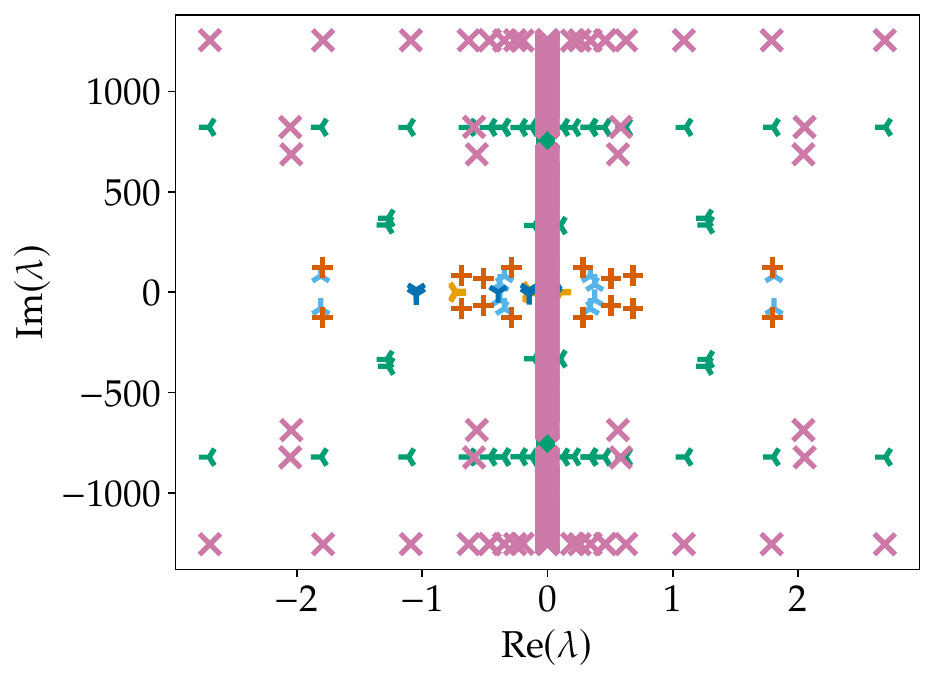}
    \end{subfigure}%
    \hspace*{\fill}
    \begin{subfigure}{0.49\textwidth}
    \centering
      \includegraphics[width=\textwidth]{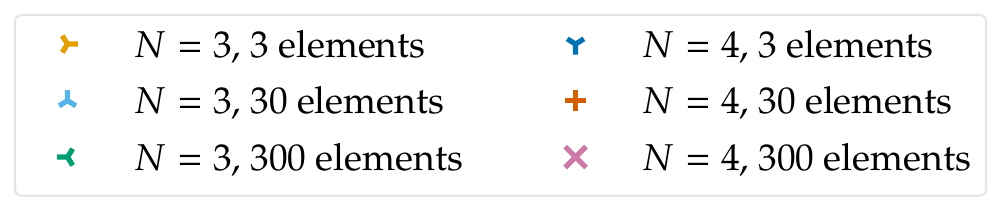}
    \end{subfigure}%
    \caption{Continuous Galerkin methods using polynomials of degree
             $N \in \{3, 4\}$ on Lobatto Legendre bases and different numbers
             of elements.}
    \label{fig:spectra-CG}
  \end{subfigure}%
  \\
  \begin{subfigure}{\textwidth}
    \centering
    \begin{subfigure}{0.49\textwidth}
    \centering
      \includegraphics[width=\textwidth]{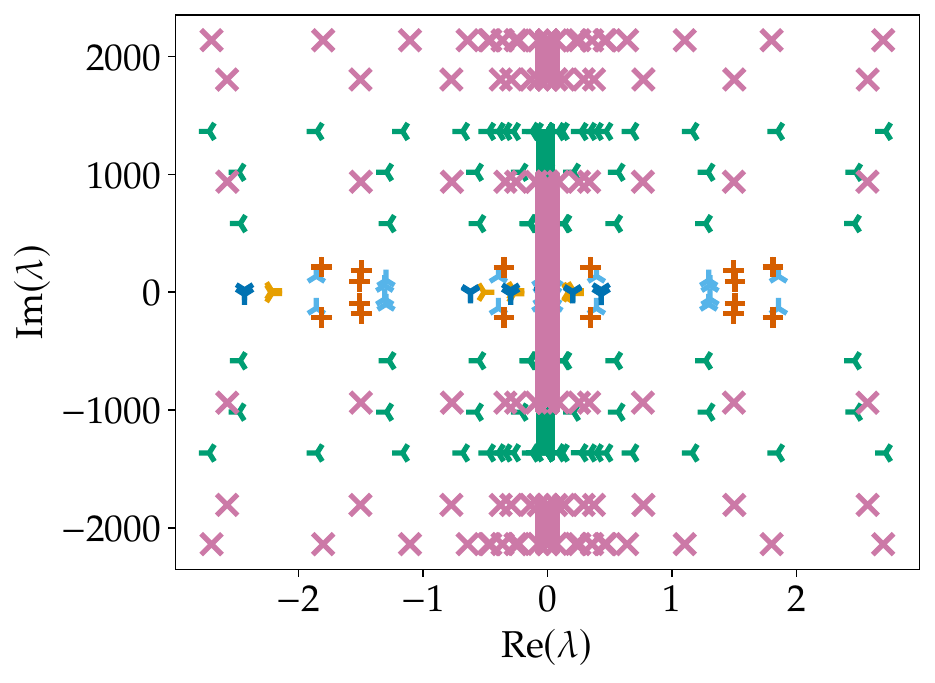}
    \end{subfigure}%
    \hspace*{\fill}
    \begin{subfigure}{0.49\textwidth}
    \centering
      \includegraphics[width=\textwidth]{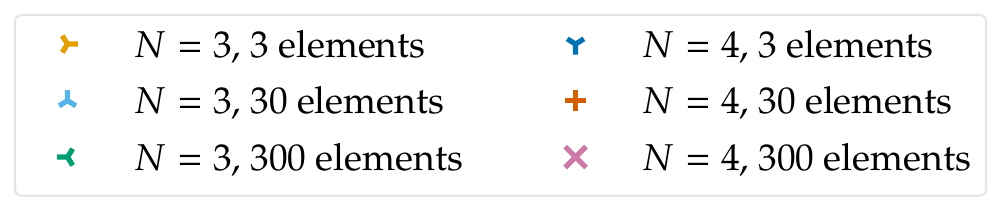}
    \end{subfigure}%
    \caption{Discontinuous Galerkin methods using polynomials of degree
             $N \in \{3, 4\}$ on Lobatto Legendre bases and different numbers
             of elements.}
    \label{fig:spectra-DG}
  \end{subfigure}%
  \caption{Spectra of nonlinear semi-discretizations \eqref{eq:cl-1D-SBP}
           using different variants of SBP operators and the logarithmic mean
           as numerical flux.}
  \label{fig:spectra}
\end{figure}
Clearly, all of the spectra have eigenvalues with positive real part (of order unity)
that do not converge to zero under grid refinement. In particular, these
eigenvalues with positive real part occur for all choices of semi-discretizations.
We checked that the occurrence of eigenvalues with positive real parts does
not depend on the parity of the number of nodes/elements or the
polynomial degree. For the investigation of spectra based on
discretizations with other mean value choices, we refer to
Appendix~\ref{sec:appendix2}. All other mean values tested give
discretizations where the spectra have significant positive real parts.

This nicely suggests that the reason
for the stability issues of the high-order split-form DG scheme with the
EC flux of Ranocha for the simple density wave example is due to the
logarithmic mean of the density in the density flux. Even with the PEP
property, which guarantees that pressure and velocity stay constant
throughout the simulation, the discretization of the density evolution
is unstable when using the logarithmic mean, while the discretization
with the flux of Shima \etal is based on an arithmetic mean of the
density and hence runs the example robustly. It remains to discuss
however, if the Shima \etal flux is locally linearly stable as defined in
\cite{gassner2020stability}, i.e., if the spectrum of the linearized
operator is stable towards perturbations.

\subsection{Investigation of local linear stability}

We want to dig deeper and analyze the respective spectra of the
Jacobians of the different DG semi-discretizations for the
two-dimensional compressible Euler equations. We observed in
Section~\ref{sec:num_inv} that the EC scheme immediately crashes
whereas the central scheme and the scheme powered by the Shima
\etal flux run for very long times ($t > 100$) without any problems.

For the linearization, we use the initial condition as the linearization
state and compute the Jacobians approximately with a central finite
difference approach in our simulation framework
\trixi \cite{schlottkelakemper2020trixi}. The resulting spectra for the
central flux with arithmetic means, the flux by Shima \etal, and
Ranocha's EC flux are shown in Figure~\ref{fig:spectra-Euler}.
As in Section~\ref{sec:num_inv}, we use the same numerical flux for the
volume terms and the surface terms without any further dissipation.
\begin{figure}[!htp]
\centering
  \begin{subfigure}{0.33\textwidth}
  \centering
    \includegraphics[width=\textwidth]{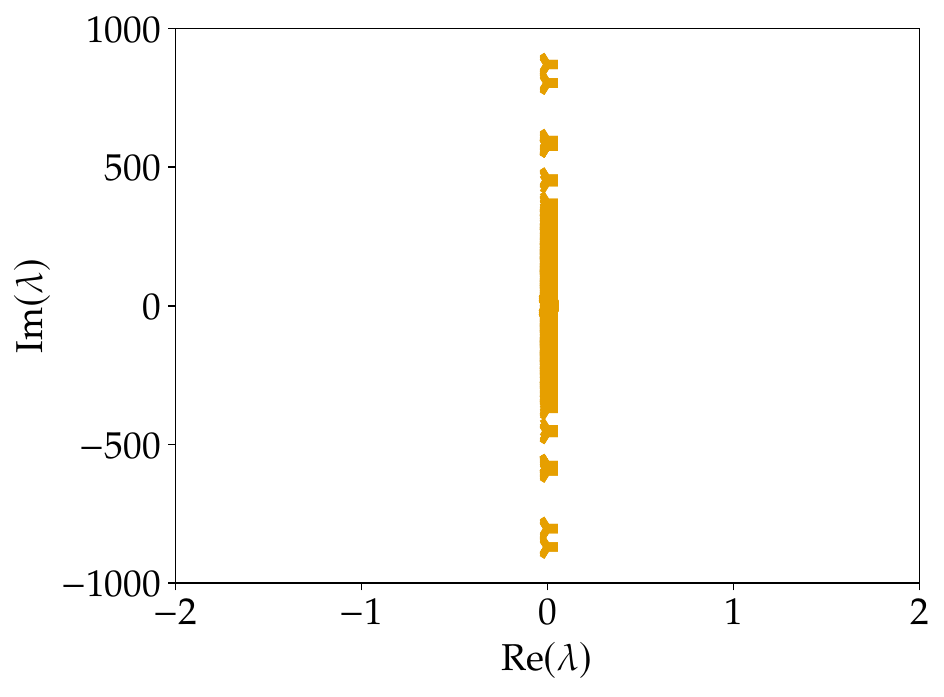}
    \caption{Central flux.}
  \end{subfigure}%
  \hspace*{\fill}
  \begin{subfigure}{0.33\textwidth}
  \centering
    \includegraphics[width=\textwidth]{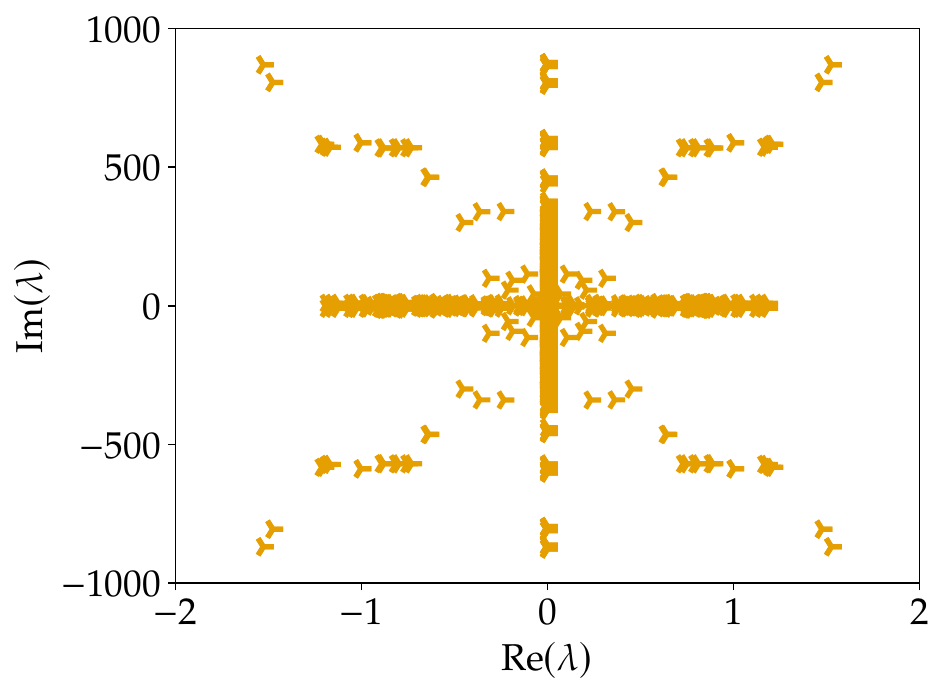}
    \caption{KEP and PEP flux of \cite{shima2021preventing}.}
  \end{subfigure}%
  \hspace*{\fill}
  \begin{subfigure}{0.33\textwidth}
  \centering
    \includegraphics[width=\textwidth]{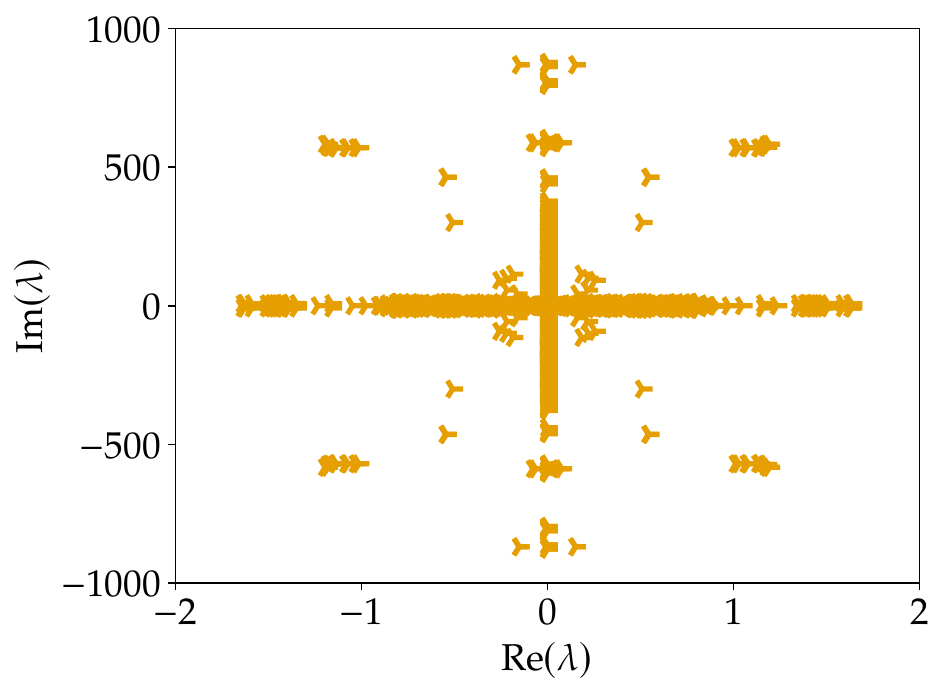}
    \caption{EC, KEP, and PEP flux of \cite{ranocha2018thesis,ranocha2020entropy}.}
  \end{subfigure}%
  \caption{Spectra of split-form DG semi-discretizations of the compressible
           Euler equations in two space dimensions with linearization state 
           \eqref{eq:Euler-density-wave}. The DG methods use polynomials of
           degree $N = 5$ and a uniform grid of $4 \times 4$ elements in the
           domain $[-1, 1]^2$ with periodic boundary conditions.
           The same numerical flux is used for both the volume terms
           and at surfaces.}
  \label{fig:spectra-Euler}
\end{figure}

As expected, the central flux results in a spectra that is almost purely
imaginary, with only small deviations of the eigenvalues from the
imaginary axis
that are within machine accuracy
when considering the approximation of the Jacobian via finite differences
and the conditioning of the associated eigenvector matrix for this problem.

\begin{remark}
  We would like to stress that eigenvalues of nonlinear right-hand sides in an
  ODE $u'(t) = g(u(t))$ do not necessarily predict the global behavior of solutions.
  For example, an energy-conserving ODE with purely positive and negative
  eigenvalues is discussed in \cite{ranocha2020strong,ranocha2020energy}.
  However, eigenvalues of the linearized Jacobian predict the local behavior
  of the solution, \eg the temporal development of initial perturbations.
  In fluid dynamics, it is well known that there are many flow states that are
  physically unstable, \ie flow states such as shear layers where added initial
  perturbations grow exponentially in time, until they start to behave non-linearly
  and transition to turbulence. In the considered case of density propagation
  discussed in this section,
  we do not expect to find significant
  signs of exponential growth, as also indicated by the almost
  imaginary spectrum obtained for the central flux. In particular, the compressible
  Euler equations are reduced to linear advection equations and hence this
  particular flow state is physically stable with respect to density perturbations.
\end{remark}

As anticipated following the numerical investigations in Section~\ref{sec:num_inv},
the EC, KEP, and PEP flux \eqref{eq:ec-kep-pep-fnumrho} of Ranocha yields
eigenvalues with clearly positive real parts of order unity, which do not
vanish under grid refinement, but shift to higher imaginary values (see
\cite{gassner2020stability} for a more detailed discussion on this effect).
This underlines our conclusion to research question (RQ2), that the PEP property
does not fix the local linear stability issue of the EC split-form DG scheme.

Surprisingly, the spectrum with the KEP and PEP flux \eqref{eq:shima_etal}
of Shima \etal \cite{shima2021preventing} shows similar issues --- it
clearly has eigenvalues with positive real parts. \emph{This
discretization is not locally linearly stable neither}. However, as
observed in Section~\ref{sec:num_inv}, the discretization could
robustly handle the density wave example (we made sure to test
very long times $t > 100$ as well).

It is important to point out that the density wave example is a specific
test case particular well suited to the Shima \etal powered
discretization where it works perfectly fine,
as it preserves the pressure
and velocity as constants down to machine precision and hence reduces to the
central scheme in this particular case.
The spectrum clearly shows that adding just a small perturbation to this state may
lead to spurious exponential growth --- hence it is not locally linearly
stable. The spectrum with the central flux is (almost) purely imaginary
and has no growth of any modes. Constant pressure and velocity
within machine precision accuracy cause very small perturbations in
the range of $10^{-15}$. Hence, it would take a really (really) long
simulation run time until these machine accuracy fluctuations grow.
Furthermore, at this very small perturbation scale, the artificial
dissipation of the time integration is effective and the $\mathtt{cfl}$
would have to be drastically reduced.

We thus need a feasible setup to further assess the robustness of the
DG split-form with the Shima \etal flux: we make a simulation that
investigates the growth of medium scale perturbations added to the
initial conditions (see \cite{gassner2020stability} for additional
details). We start with the same setup as above and compute the
eigenvector $\tilde u_0$ associated with the biggest real eigenvalue
of the semi-discretization using the numerical flux
\eqref{eq:shima_etal} of Shima \etal \cite{shima2021preventing}.
The numerically computed
eigenvector is purely real-valued. Note that this eigenvector can
always be chosen to be real-valued since the Jacobian and its
corresponding eigenvalue are both real-valued.
We normalize the eigenvector such that $\| \tilde u_0 \|_\infty = 1$
and use the perturbed initial condition $u_0 + 10^{-3} \tilde u_0$,
where $u_0$ is the original initial condition given by \eqref{eq:Euler-density-wave}.
Thus, the perturbation scale is now $10^{-3}$, instead of $10^{-15}$.
Further, the shape of the perturbation corresponds to the eigenmode
of the spectrum. Thus, we are able to compare the growth of the
fluctuations from the simulation, to the one predicted by the spectra
using the real part of the eigenvalue as the growth rate.

To get the evolution of the perturbation, we subtract in each Runge-Kutta
stage the semi-discretization applied to the unperturbed initial
condition from the resulting semi-discretization of the perturbed
initial state. We perform these numerical  experiments using the flux
\eqref{eq:shima_etal} of Shima \etal \cite{shima2021preventing}
as surface flux for the DG scheme, and in addition also using the
dissipative HLL flux \cite{harten1983upstream} as surface flux, while
both discretizations use the Shima \etal flux for the split-form volume
integral.

\begin{figure}[htp]
\centering
  \begin{subfigure}{0.49\textwidth}
  \centering
    \includegraphics[width=\textwidth]{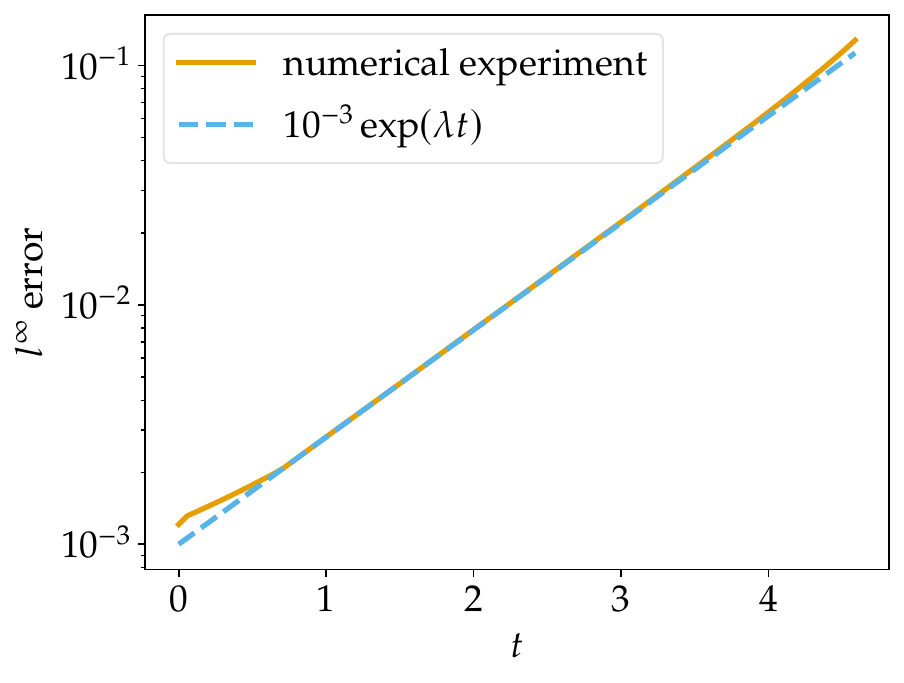}
    \caption{Shima \etal \eqref{eq:shima_etal} surface flux,
             $\lambda \approx 1.03$.}
  \end{subfigure}%
  \hspace*{\fill}
  \begin{subfigure}{0.49\textwidth}
  \centering
    \includegraphics[width=\textwidth]{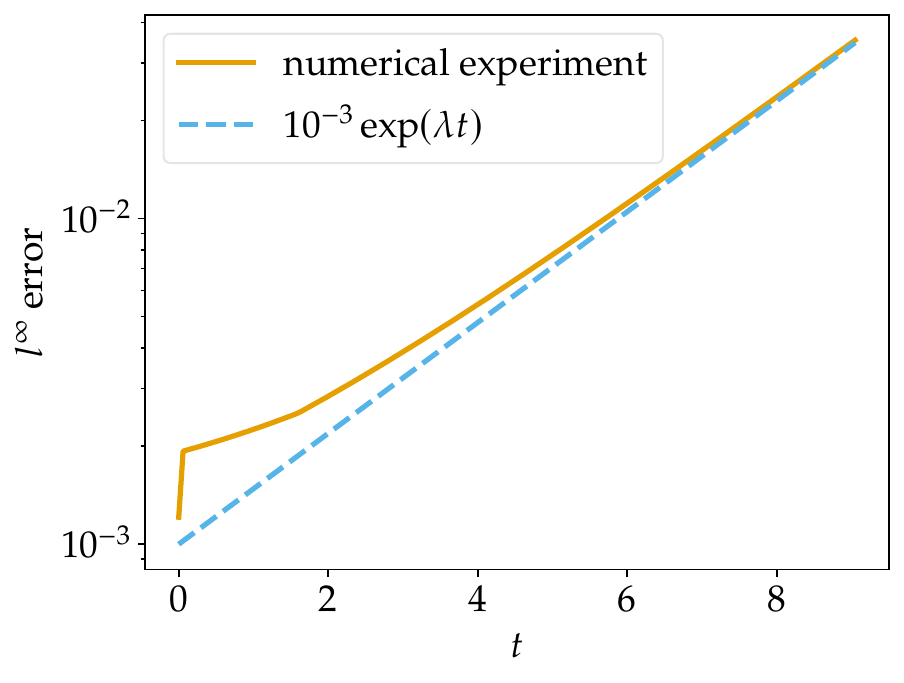}
    \caption{HLL surface flux,
             $\lambda \approx 0.39$.}
  \end{subfigure}%
  \caption{Evolution of an eigenvector perturbation of the initial condition
           \eqref{eq:Euler-density-wave}. The DG methods use polynomials of
           degree $N = 5$ and a uniform grid of $4 \times 4$ elements in the
           domain $[-1, 1]^2$ with periodic boundary conditions and the
           numerical flux of Shima \etal \cite{shima2021preventing}.}
  \label{fig:Euler-fluctuations}
\end{figure}

The resulting discrete $l^\infty$ error of the perturbations in the
conserved variables is visualized
in Figure~\ref{fig:Euler-fluctuations}.
Clearly, the fluctuations grow exponentially with a rate perfectly matching that of the real part of the eigenvalue.
The simulation terminates at $t \approx 4.6$ because of negative densities
for the case without surface dissipation and at $t \approx 9$ if the HLL flux is used.
Clearly, the dissipative HLL flux reduces the growth of the fluctuations but only
quantitatively, not qualitatively. Surface dissipation can not guarantee to control errors stemming from badly discretized split-form volume integrals. Consequently, in this case, the error still spuriously grows exponentially in time and finally results in unphysical solutions, which underlines the local stability issues of the Shima \etal flux and hence confirms the statement in \cite{gassner2020stability}, that many split-form discretizations have these problems.

\begin{figure}[htp]
\centering
  \begin{subfigure}{0.49\textwidth}
  \centering
    \includegraphics[width=\textwidth]{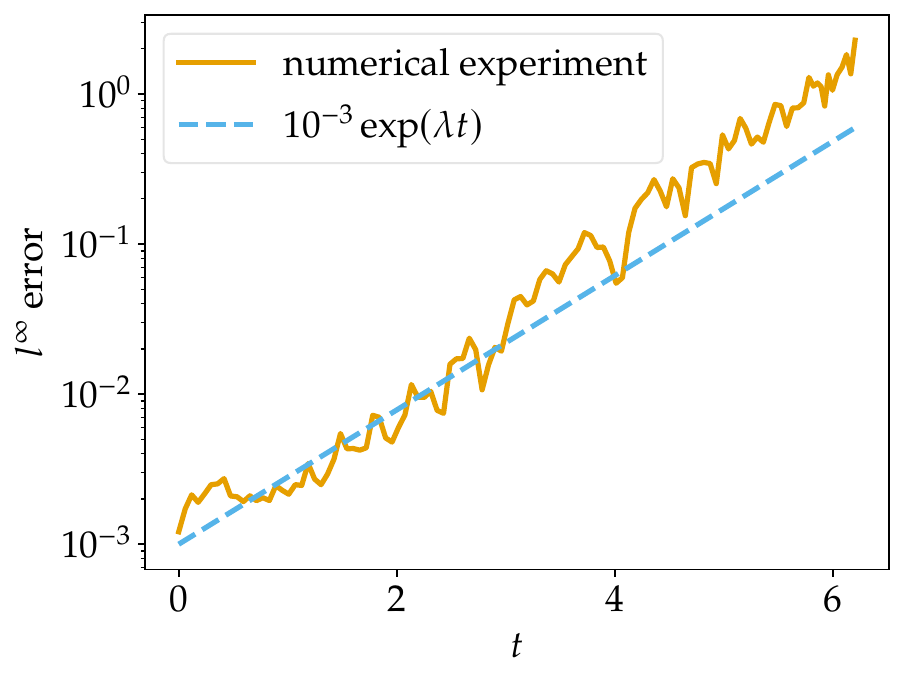}
    \caption{Shima \etal \eqref{eq:shima_etal} surface flux,
             $\lambda \approx 1.03$.}
  \end{subfigure}%
  \hspace*{\fill}
  \begin{subfigure}{0.49\textwidth}
  \centering
    \includegraphics[width=\textwidth]{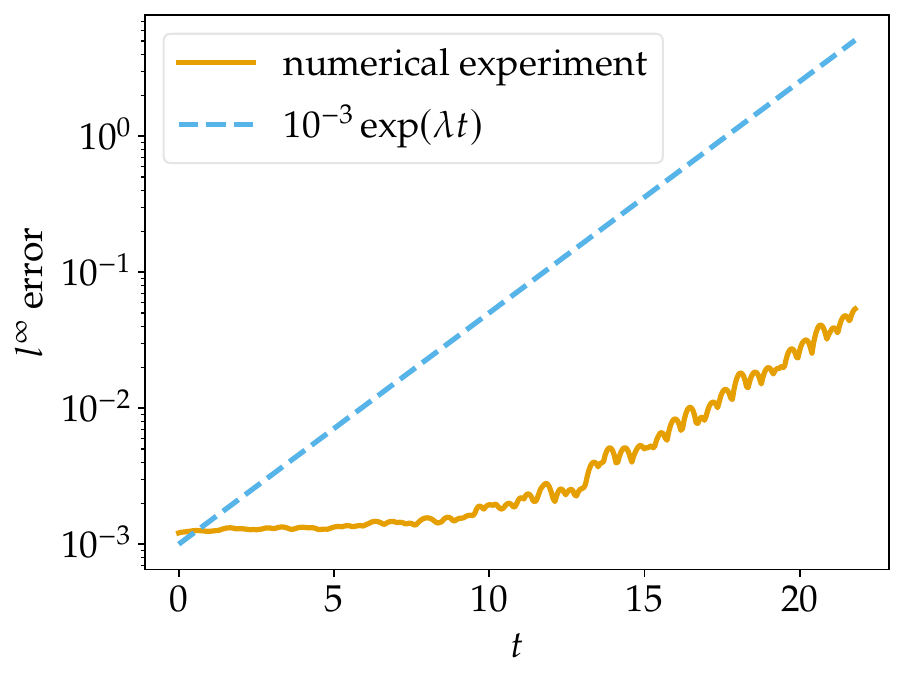}
    \caption{HLL surface flux,
             $\lambda \approx 0.39$.}
  \end{subfigure}%
  \caption{Evolution of a random perturbation of the initial condition
           \eqref{eq:Euler-density-wave}. The DG methods use polynomials of
           degree $N = 5$ and a uniform grid of $4 \times 4$ elements in the
           domain $[-1, 1]^2$ with periodic boundary conditions and the
           numerical flux of Shima \etal \cite{shima2021preventing}.}
  \label{fig:Euler-fluctuations-random}
\end{figure}

We also considered a random perturbation of the initial condition, where
each conserved variable is perturbed randomly at each point with a uniform
distribution that is symmetric around zero. Such an approach is also used to
estimate the Lyapunov exponent of a dynamical system \cite{vasil2016validated}.
The resulting errors of the fluctuation
simulation are visualized in Figure~\ref{fig:Euler-fluctuations-random}. Without
adding dissipation, the error grows approximately exponentially with a rate governed
by the maximal real eigenvalue. When surface dissipation in form of an HLL flux
is added, the perturbation grows slowly at first but shows the same exponential
growth governed by the largest real eigenvalue later.

\subsection{EC and local linear stability}

Combining the results from our numerical investigations with Corollary~\ref{cor:EC-KEP-PEP-arithmetic} or
Corollary~\ref{cor:EC-arithmetic} gives at least a partial answer to our third research question \ref{RQ3}: There are no Harten entropies for the compressible Euler equations such that the associated EC two-point fluxes result in locally linearly stable schemes.

\section{Summary and discussion}
\label{sec:summary}

In this paper, we investigate the answers to the three research questions
\begin{enumerate}[label=(RQ\arabic*), labelwidth=\widthof{\ref{RQ3}}, leftmargin=!]
  \item
  Are there EC two-point fluxes for the compressible Euler
  equations with the KEP and PEP property?
  \item
  Does the PEP property remedy the local linear stability issues of high-order
  split-form DG schemes?
  \item
  Are there entropies, such that the EC two-point fluxes and corresponding EC volume integral terms are
  locally linearly stable?
\end{enumerate}

We first stress and clarify, as discussed in detail in \cite{gassner2020stability},
the final goal is not to construct a discretization that is entropy-conservative.
However, the entropy-conservative volume terms in the high-order split-form DG discretization are the key to achieve provably discrete entropy-dissipation. And while it is possible (and common) to
introduce dissipation through the surface integrals by proper choice of
entropy-dissipative/stable two-point fluxes as surface fluxes, it was also observed here and in
\cite{gassner2020stability} that surface type dissipation alone is not enough to fix the issue stemming from the volume terms. Consequently, the answers that we found for our research questions not only underline the worrisome findings in \cite{gassner2020stability}, but even strengthen them.

Collecting the answers to our research questions, we found in Section~\ref{sec:fluxes} that the answer to the first research question \ref{RQ1} is \emph{yes}. It turns out that the flux developed in
\cite{ranocha2018thesis,ranocha2020entropy} is indeed EC, KEP, and PEP. It is interesting to note, that this is also the only available choice for the compressible Euler equations with ideal gas law.

Unfortunately, we could show that the answer to the research question \ref{RQ2} is \emph{no},
as discussed in Section~\ref{sec:stability}. The additional property of PEP does
not fix the stability issue for the simple density wave propagation when using the EC scheme. We
demonstrated that the issue is the logarithmic mean of the density, which is necessary
in the density flux. This is in contrast to the two-point flux proposed by Shima \etal
\cite{shima2021preventing}, which is KEP and PEP and can robustly run the
density wave example as it uses the arithmetic mean of density in the density flux. However, this discretization is again not locally linearly stable as shown by computing the spectra and performing simulations to analyze the growth of perturbations.

The worrisome answer to the third research question \ref{RQ3} is \emph{no}, at least if we consider the family of
entropies introduced by Harten. We could prove that it is not possible to find a Harten entropy for the
compressible Euler equation, such that the density flux is based on the
arithmetic mean of the density. Our investigations of the linear advection equation clearly show that without arithmetic mean, the discretizations are not locally linearly stable. Thus, all corresponding EC split-form schemes for the compressible Euler equation will have local linear stability issues for the simple density wave propagation.

\appendix

\section{The PEP property for high-order schemes}
\label{sec:appendix1}

Extending Lemma~\ref{lem:pep-semidiscrete}, the PEP property \eqref{eq:pep}
of a numerical flux $\fnum$ extends directly to a high-order semi-discretization
\eqref{eq:cl-1D-SBP}.
\begin{lemma}
\label{lem:pep-discrete}
  A pressure equilibrium $p \equiv \const, v \equiv \const$ is preserved by
  any general linear method applied to the semi-discretization \eqref{eq:cl-1D-SBP}
  if the numerical flux $\fnum$ is PEP.
\end{lemma}
\begin{proof}
  It suffices to consider linear combinations of numerical solutions as well as
  the addition of the semidiscrete operator to a numerical solution.
  Linear combinations preserve a pressure equilibrium, since
  $\rho$, $\rho v$, and $\rho e = \rho v^2 / 2 + p / (\gamma - 1)$
  are linear in $\rho$.
  Given $\alpha \in \R$, the scaled addition of the semidiscrete operator to
  a solution preserves the pressure equilibrium, since
  \begin{equation}
  \begin{aligned}
    \rho &\colon
    & \rho_i + \alpha \sum_l D_{i,l} \fnum_{\rho}(u_i, u_l),
    \\
    \rho v &\colon
    & (\rho v)_i + \alpha \sum_l D_{i,l} \fnum_{\rho v}(u_i, u_l)
    &=
    \rho_i v + \alpha \sum_l D_{i,l} \bigl( v \fnum_{\rho}(u_i, u_l) + \const(p, v) \bigr)
    \\&&
    &=
    \biggl( \rho_i + \alpha \sum_l D_{i,l} \fnum_{\rho}(u_i, u_l) \biggr) v,
    \\
    \rho e &\colon
    & (\rho e)_i + \alpha \sum_l D_{i,l} \fnum_{\rho e}(u_i, u_l)
    &=
    \frac{1}{2} \rho_i v^2 + \frac{1}{\gamma - 1} p +
    \alpha \sum_l D_{i,l} \biggl( \frac{1}{2} v^2 \fnum_{\rho}(u_i, u_l) + \const(p, v) \biggr)
    \\&&
    &=
    \frac{1}{2} \biggl( \rho_i + \alpha \sum_l D_{i,l} \fnum_{\rho}(u_i, u_l) \biggr) v^2 + \frac{1}{\gamma - 1} p.
  \end{aligned}
  \end{equation}
  Here, we used $\sum_l D_{i,l} = 0$, which is a necessary condition for a
  consistent derivative operator $D$.
\end{proof}

\section{Stability investigation of alternative mean values}
\label{sec:appendix2}

The logarithmic mean value is not the only mean value that is problematic for the stability. To demonstrate this, we show spectra of second-order central
finite difference methods of the form \eqref{eq:cl-1D-SBP}, where the numerical
flux is chosen as any of the different mean values studied in \cite{chen2005means},
namely
\begin{itemize}
  \item
  the centroidal mean $\fnum(u_-, u_+) = 2 (u_-^2 + u_- u_+ + u_+^2) / 3 (u_- + u_+)$,

  \item
  the arithmetic mean $\fnum(u_-, u_+) = (u_- + u_+) / 2$,

  \item
  the Heronian mean $\fnum(u_-, u_+) = (u_- + \sqrt{u_- u_+} + u_+) / 3$,

  \item
  the logarithmic mean $\fnum(u_-, u_+) = (u_+ - u_-) / (\log u_+ - \log u_-)$,

  \item
  the geometric mean $\fnum(u_-, u_+) = \sqrt{u_- u_+}$,

  \item
  and the harmonic mean $\fnum(u_-, u_+) = 2 u_- u_+ / (u_- + u_+)$.
\end{itemize}
This list is ordered in descending order of the size of the mean values
\cite{chen2005means}.

The resulting spectra are shown in Figure~\ref{fig:spectra-fluxes}. Clearly,
all mean values except the arithmetic mean result in eigenvalues with positive
real parts. The size of the maximal real part of the spectrum increases for
mean values that deviate more from the arithmetic mean value.

\begin{figure}[!tp]
\centering
  \begin{subfigure}{0.33\textwidth}
  \centering
    \includegraphics[width=\textwidth]{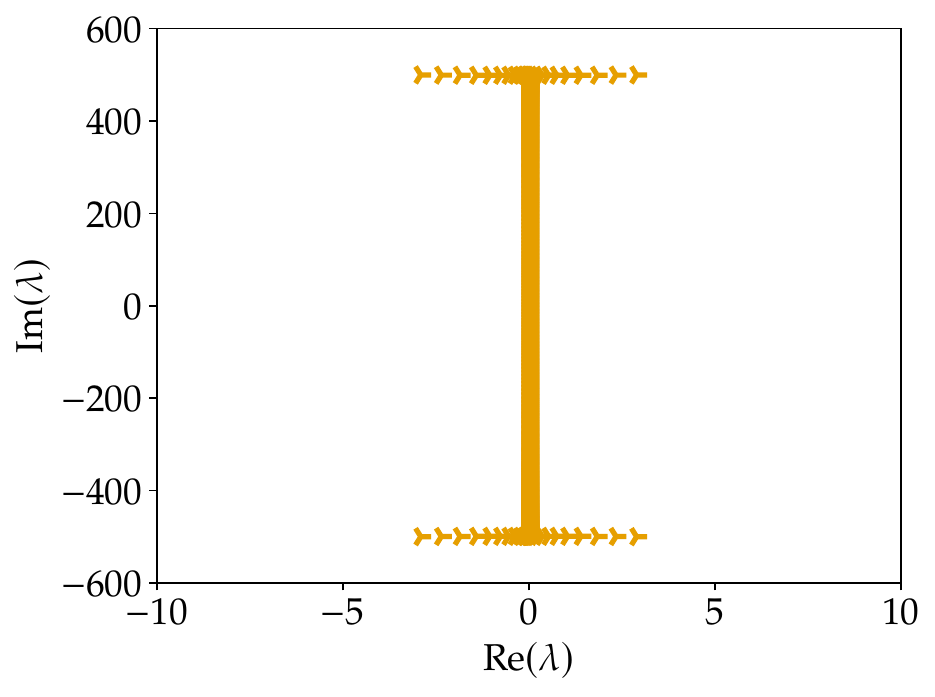}
    \caption{Centroidal mean.}
  \end{subfigure}%
  \hspace*{\fill}
  \begin{subfigure}{0.33\textwidth}
  \centering
    \includegraphics[width=\textwidth]{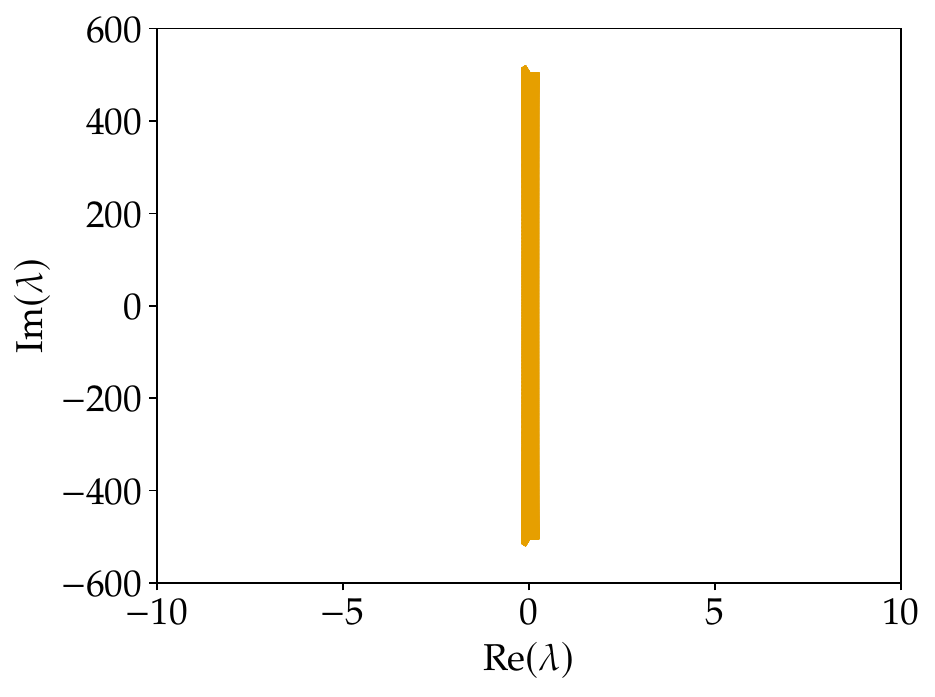}
    \caption{Arithmetic mean.}
  \end{subfigure}%
  \hspace*{\fill}
  \begin{subfigure}{0.33\textwidth}
  \centering
    \includegraphics[width=\textwidth]{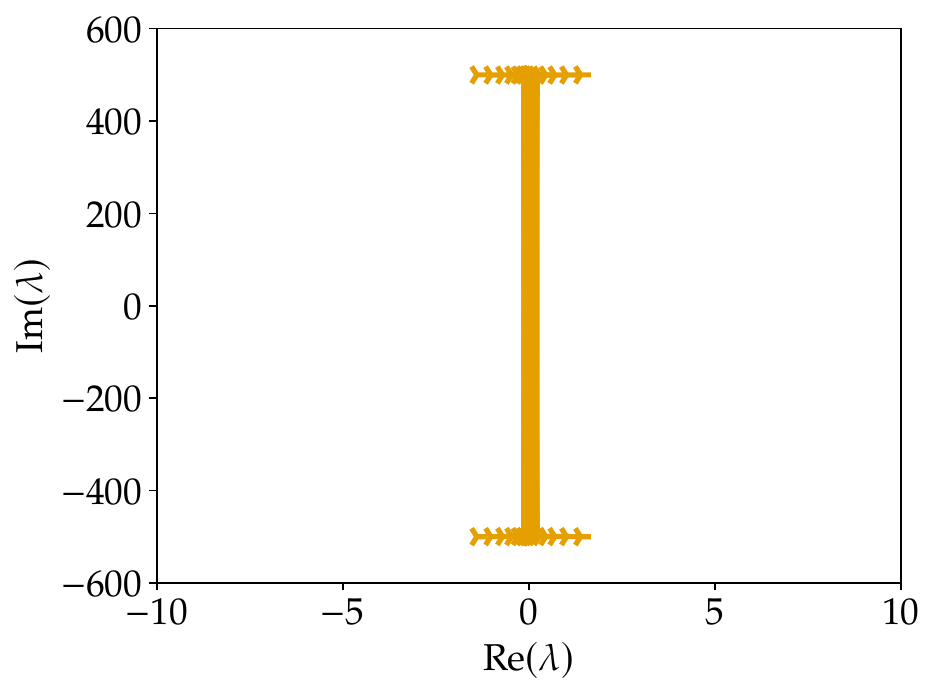}
    \caption{Heronian mean.}
  \end{subfigure}%
  \\
  \begin{subfigure}{0.33\textwidth}
  \centering
    \includegraphics[width=\textwidth]{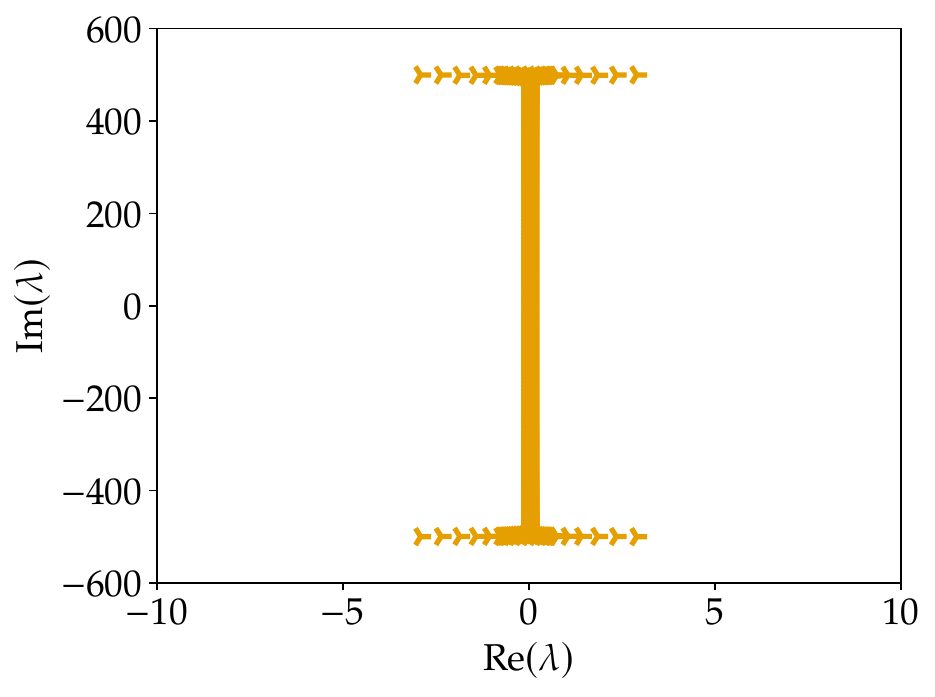}
    \caption{Logarithmic mean.}
  \end{subfigure}%
  \hspace*{\fill}
  \begin{subfigure}{0.33\textwidth}
  \centering
    \includegraphics[width=\textwidth]{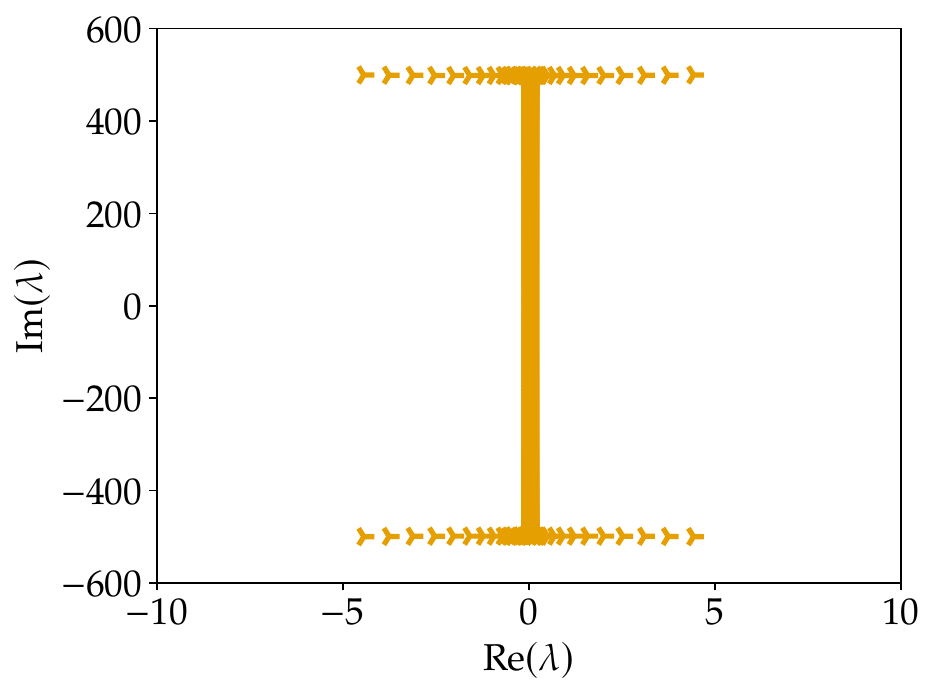}
    \caption{Geometric mean.}
  \end{subfigure}%
  \hspace*{\fill}
  \begin{subfigure}{0.33\textwidth}
  \centering
    \includegraphics[width=\textwidth]{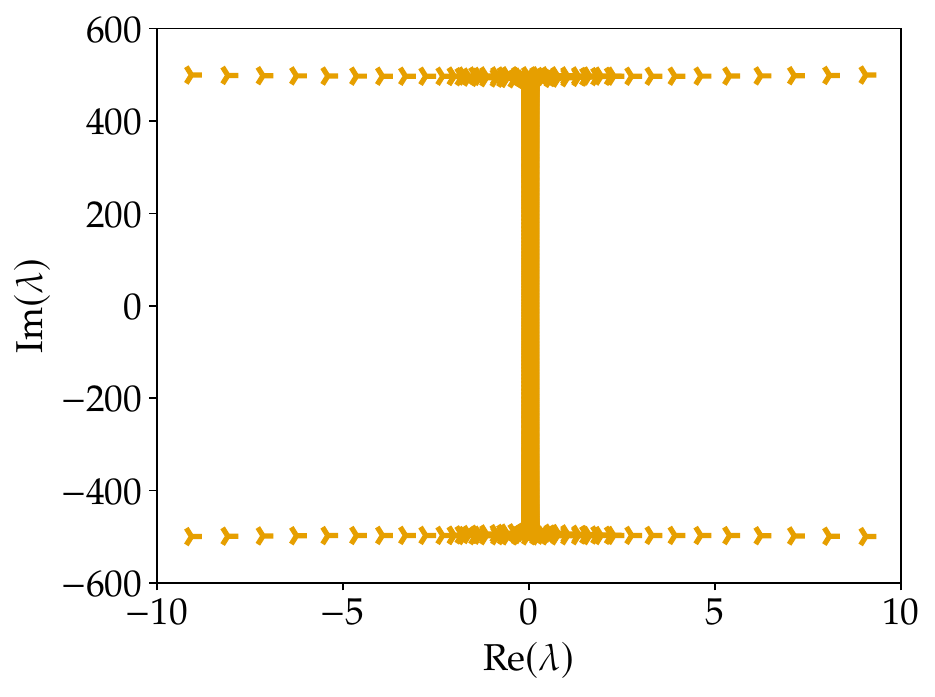}
    \caption{Harmonic mean.}
  \end{subfigure}%
  \caption{Spectra of nonlinear semi-discretizations \eqref{eq:cl-1D-SBP}
           using second-order central finite difference methods and different
           mean values as numerical flux.}
  \label{fig:spectra-fluxes}
\end{figure}

\section*{Acknowledgments}

Research reported in this publication was supported by the
King Abdullah University of Science and Technology (KAUST).
Funded by the Deutsche Forschungsgemeinschaft (DFG, German Research Foundation)
under Germany's Excellence Strategy EXC 2044-390685587, Mathematics Münster:
Dynamics-Geometry-Structure.
Gregor Gassner is supported by the European Research Council (ERC) under the
European Union's Eights Framework Program Horizon 2020 with the research project
Extreme, ERC grant agreement no. 714487.

\printbibliography

\end{document}